\RequirePackage{fix-cm}
\documentclass[smallextended]{svjour3}       
\smartqed  
\usepackage{graphicx,epstopdf,amsmath,amsfonts,amssymb,
color,varwidth,bm,bbm,nicefrac,amsfonts}
\allowdisplaybreaks[4]
\usepackage[top=1.1in, bottom=1.2in, left=1.5in, right=1.5in]{geometry}
\newtheorem{assumption}[theorem]{Assumption}

\newcommand{\R}{\mathbb{R}}
\newcommand{\N}{\mathbb{N}}
\newcommand{\E}{\mathbb{E}}
\renewcommand{\H}{\mathbb{H}}

\newcommand{\dd}{\mathrm{d}}
\begin{document}

\title{Strong convergence rates for a full discretization of stochastic wave equation with nonlinear damping\thanks{This work was supported by NSF of China (Nos. 12471394, 12071488, 12371417) and NSF of Hunan Province (No. 2020JJ2040).
This work was partially supported by STINT and NSFC Joint China-Sweden Mobility programme (project nr. CH2016-6729).
The work of D. Cohen was partially supported by the Swedish Research Council (VR) (projects nr. 2018-04443).}
}

\titlerunning{Strong approximations for a full discretization of SWE with nonlinear damping}        

\author{Meng Cai  \and David Cohen \and Xiaojie Wang 
}


\institute{Meng Cai \at
School of Statistics and Mathematics, Central University of Finance and Economics, Beijing, China\\
              \email{mcai@lsec.cc.ac.cn}           
           \and
           David Cohen \at
             Department of Mathematical Sciences,
             Chalmers University of Technology \& University of Gothenburg,
              G\"{o}teborg, Sweden\\
              \email{david.cohen@chalmers.se}
           \and
           Xiaojie Wang \at
             School of Mathematics and Statistics, HNP-LAMA, Central South University, Changsha, China\\
              \email{x.j.wang7@csu.edu.cn}
}

\date{Received: date / Accepted: date}

\maketitle

\begin{abstract}
The paper establishes the strong convergence rates of a spatio-temporal full discretization
of the stochastic wave equation with nonlinear damping in dimension one and two.
We discretize the SPDE by applying a spectral Galerkin method in space and a modified implicit exponential Euler scheme in time.
The presence of the super-linearly growing damping in the underlying model brings challenges into the error analysis.
To address these difficulties, we first achieve upper mean-square error bounds, and then obtain
mean-square convergence rates of the considered numerical solution. This is done without requiring the moment bounds of the full approximations.
The main result shows that, in dimension one, the scheme admits a convergence rate of  order $\tfrac12$ in space and
order $1$ in time.
In dimension two, the error analysis is more subtle and can be done at the expense of an order reduction due to an infinitesimal factor.
Numerical experiments are performed and confirm our theoretical findings.
\end{abstract}

\section{Introduction}\label{sec:introduction}
Hyperbolic stochastic partial differential equations (SPDEs) play an essential role in a range of real application areas,
see below for examples. In the last decade, these SPDEs have attracted considerable attention from both a theoretical and a numerical point of view.
One of the fundamental hyperbolic SPDE is the stochastic wave equation (SWE). Stochastic wave equations are
used for instance to model the motion of a vibrating string \cite{Cabana1972on} or the motion of a strand of DNA in a fluid \cite{dalang2009minicourse}.
A damping term is often included to the wave equation to model energy dissipation
and amplitude reduction, see for instance the references \cite{MR1944756,barbu2007stochastic,MR2244432,MR3078824,MR2426123,nguyen2022polynomial,pardoux1975}
on stochastic wave equations with damping.

In the present work, we focus on the strong approximation of the solution to the following SWE with nonlinear damping in a domain $\mathcal{D}$:
\begin{equation}\label{eq:SWE}
\begin{split}
\left\{
    \begin{array}{lll}
    \dd u(t) = v(t) \, \dd t,
    \\
    \dd v(t) = \left( \Delta u(t) + F ( v(t) )\right) \, \dd t + \dd W^Q(t),
    \quad  t \in (0, T],
    \\
     u(0) = u_0, \, v(0) = v_0,
    \end{array}\right.
\end{split}
\end{equation}
with homogeneous Dirichlet boundary conditions.
Here, $\Delta$ denotes the Laplace operator and
$\mathcal{D}\subset \R^d$, with $d \in \{ 1, 2\}$, is a bounded domain with smooth boundary $\partial \mathcal{D}$.
By Newton's second law, \eqref{eq:SWE} describes the displacement field of a particle suspended in a randomly continuous medium, under the impact of the stochastic force. The interactions with surrounding particles are represented by the Laplacian and a nonlinear damping force depends on the velocity.
Precise assumptions on $F$, the noise $W^Q$ on a
given probability space $(\Omega,\mathcal{F},\mathbb{P},\{\mathcal{F}_t\}_{t\geq 0})$
as well as the initial value $(u_0,v_0)$ are provided in the next section.

As opposed to the large amount of works on the numerical analysis of SPDEs of parabolic type,
see for instance the books \cite{kruse2014strong,lord2014introduction} and references therein for the globally Lipschitz setting
and \cite{becker2017strong,Arnulf2013Galerkin,brehier2019strong,cai2021weak,campbell2018adaptive,cui2019strong,feng2017finite,gyongy2016convergence,kovacs2015backward,liu2021strong,qi2019optimal,wang2020efficient} for the non-globally Lipschitz setting,
the literature on the numerical analysis of SPDEs of hyperbolic type is relatively scarce.
The numerical analysis of SWEs without damping has been investigated by several authors, see for example
\cite{anton2016full,Cao2007spectral,cohen2022numerical,cohen2013trigonometric,cohen2016fully,cui2019energy,cui2022semi-implicit,Klioba2023pathwise,kovacs2020weak,kovacs2010finite,li2022finite,quer2006space,walsh2006numerical,wang2015exponential,wang2014higher}. For instance, in the work~\cite{anton2016full},
the strong convergence rate of a fully discrete scheme for a stochastic wave equation driven by (a possibly cylindrical) $Q$-Wiener process is obtained together with an almost trace formula; Walsh in \cite{walsh2006numerical} used an adaptation of the leapfrog discretization to construct a fully discrete finite difference scheme which achieves a convergence order of $\tfrac 12$ in both time and space for an SPDE driven by space-time white noise; the authors of \cite{wang2014higher} presented strong approximations of higher order by using linear functionals of the noise process in the time-stepping schemes. For stochastic strongly damped wave equations, Qi and Wang in \cite{qi2019error} investigated the regularity and strong approximations of a full discretization performed by a standard finite element method in space and a linear implicit
Euler--Maruyama scheme in time. These authors also analyzed an accelerated exponential scheme in \cite{qi2017accelerated}.
For SWEs with weak damping, the authors of \cite{barbu2007stochastic} proved existence and uniqueness of an invariant measure.
With regard to the weak convergence analysis, we refer to the work~\cite{brehier2018weak} for a spatial spectral Galerkin approximations of a damped-driven
stochastic wave equation.
We also refer to the recent work \cite{lei2023numerical}, where the authors analyzed the approximation of the invariant distribution for stochastic damped wave equations.
Most of the aforementioned papers are concerned with SPDEs having globally Lipschitz coefficients.
For SWEs with a cubic nonlinearity, we are aware of the following references. The work \cite{schurz2008analysis} studied a nonstandard partial-implicit midpoint-type difference method to control the energy functional of the SPDE. 
The recent work \cite{cui2019energy} combines a splitting technique with the averaged vector field method,
and applies the spectral Galerkin method in the spatial direction to present an energy-preserving exponentially integrable numerical method.
Combining these structure-preserving properties with regularity estimates of the exact and the numerical solutions, the authors of the work \cite{cui2019energy}
then obtain strong convergence rates, even for low regularity noise, of the numerical scheme when applied to semilinear stochastic wave equation without a damping term.
Finally, we recall that the existence and uniqueness of an invariant measure for the underlying stochastic wave equation have been proven in \cite{barbu2007stochastic}.
An interesting question could be to investigate numerical approximations of such invariant measure, which relies on a long-time error analysis. This will be the subject of a future work.

In the present paper, we make a further contribution to the numerical analysis of SWEs with non-globally Lipschitz coefficients.
Indeed, we prove strong convergence rates of a full discretization of the SPDE~\eqref{eq:SWE}
under certain assumptions allowing for super-linearly growing coefficients. To do this, we first spatially discretize the SPDE (see the abstract equation~\eqref{eq:abstract-SPDE} below)
by a spectral Galerkin method (see equation~\eqref{eq:spectral-galerkin}).
We then propose a modified implicit exponential Euler scheme (see equation~\eqref{eq:full-discretization}) applied to the spectral Galerkin approximation. 
The main convergence results (see Theorems~\ref{thm:space}~and~\ref{thm-main2} below)
show that, under Assumptions~\ref{assum:nonlinearity} and~\ref{assum:noise-term} and further technical conditions,
the proposed fully discrete scheme strongly converges with order
$\tfrac 12$ in space and $1$ in time.
More precisely, let $X(t)$ and $X_{N,m}$
be the solution of the SWE with nonlinear damping~\eqref{eq:abstract-SPDE} and of the full approximation solution~\eqref{eq:full-discretization}. For $N,M \in \N$ and $d=1$,
there exists a constant $\widehat C > 0$, independent of the discretization parameters, such that
\begin{equation*}
\sup_{ 0 \leq m \leq M}
\big\| X(t_m) - X_{N,m}  \big\|_{L^2 (\Omega;\H^1) }
\leq \widehat C \lambda_N^{ -\frac12 } + \widehat C \, \tau,
\end{equation*}
where $\tau=\tfrac TM$ is the time stepsize, $\lambda_N$ is the $N$-th eigenvalue of the Laplacian, and the product space $\H^1$ is defined in the next section.
In addition, for $d=2$, one gets the error estimates
\begin{equation*}
\sup_{ 0 \leq m \leq M}
\big\| X(t_m) - X_{N,m}  \big\|_{L^2 (\Omega;\H^1) }\leq \widehat C \lambda_N^{ -\frac12 + \epsilon } +
\widehat C \, \tau^{1-\epsilon}.
\end{equation*}
Here, $\epsilon > 0$ is an arbitrary small parameter.

The error analysis for the space dimension $d = 3$ or higher is non-trivial. In particular, it is limited by the regularities of the mild solution $X(t)$ and spatial semi-dicretization $X^N(t)$.

We now illustrate the main steps behind the proofs of our convergence results.
For the spatial convergence analysis, we start by introducing an auxiliary process $\widetilde{X}^N(t)$ given by \eqref{eq:auxiliary-process-Galerkin} and then separate the strong error in space into two terms,
\begin{equation*}
\begin{split}
\| X(t)  - X^N(t) \|_{L^2(\Omega;\H^1)}
 & \leq  \| X(t) - \widetilde{X}^N(t) \|_{L^2(\Omega;\H^1)}
       +  \| \widetilde{X}^N(t) - X^N(t) \|_{L^2(\Omega;\H^1)}
\\ &
=: Err_1 + Err_2.
\end{split}
\end{equation*}
The bound for the term $Err_1$ can be done by a standard approach.
The estimate of the second term $Err_2$ is not easy and heavily relies on the global monotonicity property of the nonlinearity, Gronwall's lemma, suitable uniform moment bounds for the auxiliary process $\widetilde{X}^N$ and the numerical approximations as well as the bounds of $Err_1$.
For the temporal convergence analysis, motivated by the approach from the work~\cite{wang2020meanEuler}
for finite-dimensional stochastic differential equations, we first show an upper bound of the temporal error
in Proposition~\ref{prop:upper-bounds}. This result then enables us to prove mean-square convergence rates for the considered SPDE, without requiring an a priori high-order moment estimates of the fully discrete solution.

It is worthwhile to mention that the error estimates for dimension two
is more involved than for dimension one since the Sobolev embedding inequality $\dot{H}^1 \subset V:=C(\mathcal D; \mathbb R)$
fails to hold in dimension two.
In order to overcome this difficulty, we combine  H\"{o}lder's inequality and the Sobolev embedding theorem
to bound the nonlinearity $F(v)$ in a weak sense (see Lemma~\ref{lem:F-theta} below).
This causes an order reduction in the rate of convergence due to an infinitesimal factor in the convergence analysis for $d=2$.

The outline of the paper is as follows.
We start by collecting some notation, useful results and then show
the well-posedness of the considered problem in Section~\ref{sec:preliminaries}.
Section~\ref{sec:spatial-analysis} and Section~\ref{sec:full-analysis} are devoted to the mean-square convergence analysis in space and time, respectively. Numerical experiments are presented in
Section~\ref{sec;Numerical experiments} and illustrate the obtained convergence rates.
The paper ends with a conclusion.
\section{The stochastic wave equation with nonlinear damping}
\label{sec:preliminaries}

In this section, we set notation and show the well-posedness of the stochastic wave equation with nonlinear damping.

Consider two separable Hilbert spaces $U$ and $H$ with norms denoted by
$\| \cdot \|_U$ and $\| \cdot \|_H$ respectively. We denote by $\mathcal{L}(U;H)$ the space of bounded linear operators
from $U$ to $H$ with the usual operator norm $\| \cdot \|_{\mathcal{L}(U;H)}$.
As an important subspace of $\mathcal{L}(U;H)$, we let $\mathcal{L}_2(U;H)$ be the set of Hilbert--Schmidt operators with the norm
\begin{equation*}
\| T \|_{\mathcal{L}_2(U;H)}
:=\Big( \sum_{k=1}^\infty  \|  Te_k  \|_H^2 \Big)^{\frac12},
\end{equation*}
where $\{e_k\}_{k=1}^\infty$ is an arbitrary orthonormal basis of $U$.
If $H=U$, we write $\mathcal{L}(U):=\mathcal{L}(U;U)$ and $\mathcal{L}_2 (U):=\mathcal{L}_2(U;U)$ for short.
Let $Q \in \mathcal{L}(U)$ be a self-adjoint, positive semidefinite operator. As usual,
we introduce the separable Hilbert space $U_0 :=Q^{\frac12}(U)$ with the inner product
$ \langle u_0, v_0 \rangle_{U_0}:= \langle Q^{-\frac12} u_0,Q^{-\frac12} v_0  \rangle_U$.
Furthermore, the set $\mathcal{L}_2^0$ denotes the space of Hilbert--Schmidt operators from $Q^{1/2}(U)$ to $U$
with norm $ \| T \|_{\mathcal{L}_2^0} =  \| TQ^{1/2}  \|_{\mathcal{L}_2 (U)}$. Finally,
let $(\Omega,\mathcal{F},\mathbb{P},\{\mathcal{F}_t\}_{t\geq 0})$ be a filtered probability space
and $L^p(\Omega;U)$ be the space of $U$-valued integrable random variables with norm
$
\| u \|_{L^p(\Omega;U)}
    := \big( \E  [ \| u \|_U^p] \big)^{\frac 1p} < \infty
$
for any $ p \geq 2 $.

In the sequel, we take  $U:=L^2( \mathcal{D};\R)$ with norm  $\| \cdot \|$ and inner product $\langle \cdot,\cdot \rangle$.
We also set $V:=C(\mathcal{D}; \R)$ to be the Banach space of all continuous functions endowed with the supremum norm.
We let $-\Lambda:=\Delta$ denote the Laplacian with ${\rm{Dom}}(\Lambda) = H^2(\mathcal{D}) \cap H^1_0(\mathcal{D})$.
Here, $ H^m(\mathcal{D}) $ denotes the standard Sobolev spaces of integer order $m \geq 1$.
For $\alpha\in\mathbb{R}$, we then define the separable Hilbert spaces
$\dot{H}^\alpha={\rm{Dom}}(\Lambda^{\frac \alpha 2})$ equipped with inner product
\begin{equation*}
\langle u , v \rangle_\alpha
:= \langle \Lambda^{\frac \alpha 2} u ,
      \Lambda^{\frac \alpha 2} v \rangle
=\sum_{j=1}^\infty
\lambda_j^\alpha
\langle  u , \varphi_j \rangle
\langle  v , \varphi_j \rangle,
\quad u,v \in \dot{H}^\alpha,
\end{equation*}
where $\{(\lambda_j,\varphi_j)\}_{j=1}^\infty$ are the eigenpair of $\Lambda$ with $\{\varphi_j\}_{j=1}^{\infty}$
being orthonormal eigenfunctions. The corresponding norm in the space $\dot H^\alpha$ is defined by
$\| u \|_{\alpha} = \langle u , u \rangle_\alpha ^{\nicefrac12}$. Furthermore, we introduce the product space
$
\H^\alpha: =\dot{H}^\alpha  \times  \dot{H}^{\alpha-1}
$
with the inner product
\begin{equation*}
\langle Y , Z \rangle_{\H^\alpha } :=
      \langle y_1 , z_1 \rangle_{ \alpha }
         +  \langle y_2 , z_2 \rangle_{ \alpha - 1 }
         \quad \text{for} \quad Y=\bigl[y_1,y_2\bigr]^T \quad\text{and}\quad Z=\bigl[z_1,z_2\bigr]^T.
\end{equation*}
The induced  norm is denoted by
$\| X \|_{\H^\alpha } := \bigl( \| x_1 \|_\alpha^2
      + \| x_2 \|_{\alpha-1}^2 \bigr)^{\nicefrac12}$
for $X=[x_1,x_2]^T$.
For the special case $\alpha = 0$, we define
$\H:=\H^0=\dot H^0\times\dot H^{-1}$
and $\dot{H}^0 = U =L^2( \mathcal{D} ; \R )$.

To follow the semigroup framework of the book \cite{da2014stochastic},
we formally transform the SPDE~\eqref{eq:SWE} into the following abstract Cauchy problem:
\begin{equation}\label{eq:abstract-SPDE}
 \dd X(t) = A X(t) \, \dd t + {\bf F}(X(t)) \, \dd t + B \dd W^Q(t),
 \quad
 X(0) = X_0,
\end{equation}
where
\begin{equation*}
X =X(t)=\left[  \begin{array}{ccc}
     u(t) \\  v(t)   \end{array}   \right]
,
\:
A= \left[  \begin{array}{ccc}   0 & I \\  -\Lambda & 0
                       \end{array} \right]
,
\:
{\bf F}(X)= \left[ \begin{array}{ccc} 0 \\ F(v) \end{array} \right]
,
\:
B = \left[ \begin{array}{ccc}  0 \\  I   \end{array} \right],
\:
X_0 =
 \left[
 \begin{array}{ccc}
     u_0 \\ v_0
 \end{array}
 \right].
\end{equation*}
Here, $X_0$ is assumed to be an $\mathcal{F}_0$-measurable random variable.
The operator $A$ with
$\text{Dom}(A)= \H^1 = \dot{H}^1 \times \dot{H}^0$
is the generator of a strongly continuous semigroup
$\left(E(t)\right)_{t \geq 0}$ on $\H^1$, written as
\begin{equation*}
E(t)= e^{t A} =
\left[
 \begin{array}{ccc}
     C(t) &
     \Lambda^{-\frac12} S(t) \\
     -\Lambda^{\frac12} S(t) &
     C(t)
 \end{array}
  \right],
\end{equation*}
where
$C(t) := \cos( t \Lambda^{\frac12})$
 and $S(t) := \sin( t \Lambda^{\frac12})$
 are the so-called cosine and sine operators. These operators are bounded in the sense that
 $ \| C(t) \varphi \| \leq \| \varphi \|$ and
 $ \| S(t) \varphi \| \leq \| \varphi \|$ hold for
 any $ \varphi \in U$.
The trigonometric identity $\big\| S(t)  \varphi \big\|^2 +
         \big\| C(t)  \varphi \big\|^2
               = \big\|  \varphi \big\|^2$
also holds for any $\varphi \in U$.
Additionally, due to the commutative property between
$C(t)$, $S(t)$ and $\Lambda^{\alpha}$ for $\alpha \in \R$,
one can check the stability property of the semigroup, that is
$ \| E(t) X \|_{\H^\gamma} \leq \| X \|_{\H^\gamma}$
for $t \geq 0$, $\gamma \in \R$ and $X \in \H^\gamma$.
Finally, we recall the following lemma which will be used frequently in our convergence analysis.
\begin{lemma}\label{lem:operator-estimate}
For $t \geq s \geq 0$ and $\gamma \in [0,1]$, we have
\begin{equation}\label{eq:cosine-sine}
\big\|
\big( S(t) - S(s) \big) \Lambda^{-\frac \gamma2}
\big\|_{\mathcal{L}(U)} \leq \widehat C (t-s)^{\gamma},
\qquad
\big\|
\big( C(t) - C(s) \big) \Lambda^{-\frac \gamma2}
\big\|_{\mathcal{L}(U)} \leq \widehat C (t-s)^{\gamma}
\end{equation}
for some constant $\widehat C>0$.
Moreover, for $X\in\H^\gamma$ it holds that
\begin{equation}\label{eq:semigroup}
\big\|
\big( E(t) - E(s) \big) X
\big\|_{\H} \leq \widehat C (t-s)^{\gamma} \|X\|_{\H^{\gamma}}
\end{equation}
for some constant $\widehat C>0$.
\end{lemma}
We refer to \cite{anton2016full,MR3123856}, for instance, for a proof of this lemma.
To show the well-posedness of the SPDE~\eqref{eq:abstract-SPDE}, we make the following assumptions on the nonlinear term and
on the noise process, see \cite{barbu2007stochastic}.
\begin{assumption}\label{assum:nonlinearity}
The nonlinear term $F$ is assumed to be the Nemytskij operator associated to a real-valued function $f: \R \to \R$
given by
\begin{align}
F (v)(x)
=f ( v(x))
\quad
x \in \mathcal{D},
\end{align}
where $f$ is assumed to be twice continuous differentiable and satisfies  for $y\in \R$,
\begin{align*}
y f(y) & \leq C_0 \left( 1 + |y|^2 \right), \quad\text{for some constant}\quad C_0 > 0,
\\  f'(y) & \leq C_1, \quad\text{for some constant}\quad C_1 \in \R,
\\  |  f'(y) | & \leq C_2 \left( 1 + |y|^{\gamma - 1}\right),
\quad\text{for some constant}\quad C_2 >0, \gamma \geq 2,
\\  |  f''(y) | & \leq C_3 ( 1 + |y|^{\gamma - 2}), \quad\text{for some constant}\quad C_3 >0,\gamma \geq 2.
\end{align*}
\end{assumption}
\noindent A typical example of a nonlinearity satisfying the above assumptions is $f(y)= y - y^3$.
Indeed, standard calculations give
\begin{align*}
 y f( y ) & = y^2 - y^4 \leq   1 + y^2,
\\  f'( y ) & = 1 - 3 y^2  \leq 1,
\\  |  f'( y ) | & = | 1 - 3 y^2 | \leq 3 \left( 1 + y^{2} \right),
\\  |  f''(y) | & = 6 | y | \leq 6 \left( 1 + | y | \right).
\end{align*}
This shows Assumption~\ref{assum:nonlinearity} for $f(y)= y- y^3$.

We highlight that the above function $f$ does not satisfy a globally Lipschitz condition:
If this would have been the case, then there would exist a positive constant $L$ such that for any $y_1,y_2 \in \R$,
\begin{equation*}
\big| f ( y_1 ) - f ( y_2 ) \big| \leq L  \big| y_1 - y_2 \big|.
\end{equation*}
Taking $y_2 = 0$ would yield $\big|  y_1 - y_1^3 \big| \leq L  \big| y_1 \big|$
and thus $\big|  1 - y_1^2 \big| \leq L $ for $y_1 \in \R \setminus \{0\}$. This gives a contradiction.

\begin{assumption}\label{assum:noise-term}
Let
$\{ W^Q(t)\}_{t \in [0,T]}$
be a standard {Q}-Wiener process such that the covariance operator $Q= Q^{\frac12} \circ  Q^{\frac12}$ satisfies
\begin{align}\label{eq:ass-AQ-condition}
\big\| \Lambda^{\frac12} Q^{\frac12}
              \big\|_{\mathcal{L}_2(U)} < \infty.
\end{align}
\end{assumption}
\noindent An example of a covariance operator satisfying the above condition is $Q = \Lambda^{- \delta}$,
$\delta > 1 + \frac d2$.
To simplify the notation, we often write $\mathcal{L}_2$ instead of $\mathcal{L}_2(U)$ whenever it is clear from the context which one is meant.

The well-posedness of the SPDE~\eqref{eq:abstract-SPDE} and the spatial regularity of the mild solution $X(t)$ are given in the following theorem.
\begin{theorem}\label{thm:wellposedness-regularity-Z}
Let $T>0$. Under Assumptions  \ref{assum:nonlinearity} and \ref{assum:noise-term},
the stochastic wave equation \eqref{eq:abstract-SPDE} admits a unique mild solution in $\mathbb{H}^1$, given by, for each $ t \in [ 0, T]$,
\begin{equation}\label{eq:mild-abstract}
X(t) = E(t) X_0 + \int_0^t E(t-s) {\bf F}(X(s)) \, \dd s
                 + \int_0^t E(t-s) B \, \dd W^Q(s),
                 \, a.s.
\end{equation}
Additionally, if we assume that the initial value satisfies
$ \| X_0 \|_{L^p(\Omega; \H^2)} < \infty$ for some $p \geq 2$,
then we get the bound
\begin{equation}\label{eq:moment-bounds-Z(t)}
\sup_{t \in [0,T]} \| X(t) \|_{L^p(\Omega;\H^2)} < \infty.
\end{equation}
\end{theorem}

\begin{proof}
The existence and uniqueness of the mild solution~\eqref{eq:mild-abstract} can be proven
as in the reference \cite{barbu2007stochastic}.
We now prove the spatial regularity~\eqref{eq:moment-bounds-Z(t)}.
An application of the  It\^{o} formula for $\| X(t) \|_{\H^2}^p$,
Young's inequality, properties of stochastic integrals, the assumptions on the inital value and
on the dissipativity of the nonlinear term $F$, and the assumption~\eqref{eq:ass-AQ-condition} on the $Q$-Wiener process
yield that
\begin{align*}
\E \big[  \| X(t) \|_{\H^2}^p \big]
        & =  \E \Big[ \| X_0 \|_{\H^2}^p \Big]
    + p \int_0^t \E \Big[ \| X(s) \|_{\H^2}^{p-2}
           \langle X(s) , \dd X(s) \rangle_{\H^2} \Big]
  \\ & \quad + \frac12 \sum_{j=1}^{\infty}
  \int_0^t \E \Big[  p \| X(s) \|_{\H^2}^{p-2}
  \langle B Q^{\frac12} \varphi_j,
          B Q^{\frac12} \varphi_j \rangle_{\H^2}
  \\ & \quad    + p (p-2) \| X(s) \|_{\H^2}^{p-4}
      \langle X(s),
          B Q^{\frac12} \varphi_j \rangle_{\H^2}
      \langle X(s),
          B Q^{\frac12} \varphi_j \rangle_{\H^2}  \Big] \dd s
  \\ & \leq\widehat C +  p \int_0^t \E \Big[ \| X(s) \|_{\H^2}^{p-2}
           \langle X(s) , A X(s) \rangle_{\H^2}   \Big] \dd s
  \\ & \quad +  p \int_0^t \E \Big[ \| X(s) \|_{\H^2}^{p-2}
           \langle X(s) , {\bf F}( X(s)) \rangle_{\H^2} \Big] \dd s
  \\ & \quad +  p \int_0^t \E \Big[ \| X(s) \|_{\H^2}^{p-2}
           \langle X(s) , B\, \dd W^{Q}(s)\rangle_{\H^2} \Big]
  \\ & \quad +\widehat C(p) \int_0^t \E \Big[   \| X(s) \|_{\H^2}^{p-2}
  \| \Lambda^{\frac12} Q^{\frac12} \|_{\mathcal{L}_2}^2 \Big] \dd s
  \\ & \leq\widehat C(p,T) + p
            \int_0^t   \E \Big[ \| X(s) \|_{\H^2}^{p-2}
     \langle \nabla v(s) , F'( v(s) )\nabla v(s) \rangle \Big] \dd s
  \\ & \quad  +\widehat C(p,T) \int_0^t
        \E \big[ \| X(s) \|_{\H^2}^{p} \big] \dd s
  \\ & \leq\widehat C(p,T) +\widehat C(p,T)
      \int_0^t \E \big[ \| X(s) \|_{\H^2}^{p} \big] \dd s.
\end{align*}
Here, we used  $\langle X(s), A X(s) \rangle_{\H^2} =0$,
$\langle v(s),  F( v(s) )\rangle_{\dot{H}^1} =\langle \nabla v(s),  F'( v(s) )  \nabla  v(s)  \rangle$  and
$\| X(s) \|_{\H^2}^{p-2} \leq \| X(s) \|_{\H^2}^{p} +1 $ in the second inequality.
At last, an application of Gronwall's lemma finishes the proof.
\end{proof}

We conclude this section by introducing some basic inequalities especially useful when considering
the SPDE~\eqref{eq:abstract-SPDE} in dimension $d=2$.
Recall first the following  Sobolev embedding inequality,
(e.g., \cite[Theorem 7.57]{adams1975sobolev} and \cite[Lemma 3.1]{thomee2006galerkin}),
for sufficiently small $\epsilon > 0$ and $\theta \in (0,1)$,
\begin{equation}\label{eq:sobolev}
\dot{H}^{ 2 \epsilon } \subset L^{\frac 2 { 1 - 2 \epsilon } }, \,\,
\dot{H}^{ 1 - \theta } \subset L^{\frac 2 \theta}.
\end{equation}
Then, for $x\in L^{\frac 2 {1 + 2 \epsilon} }$, one has
\begin{equation}\label{eq:Lambda-minus-epsilon}
\begin{split}
\| \Lambda^{-\epsilon} x \| & =
  \sup_{ \| y \| = 1 } \big| \langle x , \Lambda^{-\epsilon} y
                                \rangle \big|
  \leq \sup_{ \| y \| = 1 } \| x \|_{L^{\frac 2 {1 + 2 \epsilon} }}
        \| \Lambda^{-\epsilon} y \|_{L^{\frac 2 {1 - 2 \epsilon} }}
 \\ & \leq\widehat C \sup_{ \| y \| = 1 }
        \| x \|_{L^{\frac 2 {1 + 2 \epsilon} }}
        \| \Lambda^{-\epsilon} y  \|_{2\epsilon}
      \leq\widehat C \| x \|_{L^{\frac 2 {1 + 2 \epsilon} }}.
\end{split}
\end{equation}
Concerning the nonlinear term $F$, we have the following useful lemmas.
\begin{lemma}\label{lem:F-theta}
Under Assumption~\ref{assum:nonlinearity}, we have,
for $d =1$, that
\begin{equation*}
\| F (v) \|_{1}
          \leq\widehat C \left( 1 + \|v\|_1^{\gamma} \right),
\quad
v\in\dot H^1
\end{equation*}
and for $d = 2$, $ \theta \in (0,1)$, that
\begin{equation*}
\| F (v) \|_{\theta}
          \leq\widehat C \left( 1 + \|v\|_1^{\gamma} \right),
          \quad
v\in\dot H^1.
\end{equation*}
\end{lemma}
\begin{proof}
In dimension $d=1$, using properties of the space $\dot H^\theta$, the assumption on $F$ and
the Sobolev embedding $\dot{H}^1\subset V$, one has
\begin{equation*}
\begin{split}
 \| F (v) \|_1= \| \nabla  ( F(v) ) \|
              = \| F'(v) \nabla v\|
     \leq\widehat C \Big(  1 + \big\| v
        \big\|_{V}^{\gamma-1} \Big)
            \| v \|_1
          \leq\widehat C \big( 1 + \|v\|_1^{\gamma} \big).
\end{split}
\end{equation*}
In dimension $d=2$, the Sobolev embedding inequality~\eqref{eq:sobolev} and H\"older's inequality, yield
\begin{equation*}\label{eq:Lambda-theta}
\begin{split}
\|  x \|_{\theta} & =  \| \Lambda^{\frac \theta 2}  x  \|
  =\sup_{ \| y \| = 1 } \big| \langle \Lambda^{\frac \theta 2}  x ,
                                       y  \rangle \big|
  = \sup_{ \| y \| = 1 } \big|\langle \Lambda^{\frac 12}  x ,
                  \Lambda^{\frac {\theta-1}2}    y  \rangle \big|
 \\ & \leq\widehat C \sup_{ \| y \| = 1 }
  \| \Lambda^{\frac 1 2}  x \|_{L^{\frac 2 {2 - \theta} }}
  \| \Lambda^{\frac {\theta-1}2} y  \|_{L^{\frac 2 {\theta} }}
   \leq\widehat C \| \Lambda^{\frac 1 2} x \|_{L^{\frac 2 {2 - \theta} }}.
\end{split}
\end{equation*}
As a consequence, one obtains
\begin{equation}\label{eq:F-theta}
\begin{split}
\| F (v) \|_{\theta} &
     \leq\widehat C \| F'(v) \nabla v \|_{L^{\frac 2 {2 - \theta} }}
     \leq\widehat C \| F'(v) \|_{L^{\frac 2 {1 - \theta} }}
            \| \nabla v \|
     \\ & \leq\widehat C \Big(  1 + \big\| v
        \big\|_{L^{\frac {2(\gamma-1)}{1-\theta}}}^{\gamma-1} \Big)
            \| v \|_1
          \leq\widehat C \big( 1 + \|v\|_1^{\gamma} \big),
\end{split}
\end{equation}
where in the last inequality, we used $\dot{H}^1 \subset L^{\frac {2(\gamma-1)}{1-\theta}}$ for $d=2$.
This concludes the proof of the lemma.
\end{proof}

\begin{lemma}\label{lem;F(-1)}
Under Assumption \ref{assum:nonlinearity}, we have for $d=1$, that
\begin{equation}\label{eqn:F(-1)}
\| F'(\phi) \psi \|_{-1} \leq\widehat C \,
             ( 1 + \| \phi \|_1^{\gamma - 1} ) \| \psi \|_{-1}
\end{equation}
for $\phi \in \dot{H}^1$ and $\psi \in H$.
\end{lemma}

\begin{proof}
Let us start the proof with a preliminary step.
Using the assumption on the nonlinearity and applying the Sobolev embedding inequality $\dot{H}^1 \subset V$ yield
\begin{equation*}
\begin{split}
\| \nabla\left( F' (\phi) \varphi \right)\|^2 & =
    \int_{\mathcal{D}} \Big| \tfrac{\partial}{\partial x}
         \big( f'(\phi(x)) \varphi(x) \big) \Big|^2 \dd x
    \\  & \leq 2 \int_{\mathcal{D}} \Big(
          \big| f''(\phi(x)) \phi'(x) \varphi(x) \big|^2 +
          \big| f'(\phi(x)) \varphi'(x) \big|^2 \Big) \dd x
    \\  & \leq\widehat C\big ( 1 + \| \phi \|_V^{\gamma-2} \big)^2
                           \| \varphi \|_V^2 \| \phi\|_1^2
         +\widehat C\big ( 1 + \| \phi \|_V^{\gamma-1} \big)^2 \| \varphi \|_1^2
    \\  & \leq\widehat C \, \Big( 1 + \|\phi\|_1^{\gamma-1} \Big)^2 \| \varphi \|_1^2
\end{split}
\end{equation*}
for $\phi\in\dot{H}^1$ and $\varphi\in\dot{H}^1$.

To get equation~\eqref{eqn:F(-1)}, we note that
\begin{equation*}
\begin{split}
\| \Lambda^{-\frac12} F'( \phi ) \psi \|
   & = \sup_{\| \xi \| \leq 1}
      \Big| \langle \Lambda^{-\frac12} F'( \phi ) \psi, \xi \rangle \Big|
     = \sup_{\| \xi \| \leq 1}
      \Big| \langle \Lambda^{-\frac12} \psi,
             \Lambda^{\frac12} F'( \phi ) \Lambda^{-\frac12} \xi \rangle \Big|
   \\ & \leq \| \psi \|_{-1} \cdot \sup_{\| \xi \| \leq 1}
                \| F' ( \phi ) \Lambda^{-\frac12} \xi \|_1
        \leq\widehat C \,\left( 1 + \| \phi \|_1^{\gamma-1} \right) \| \psi \|_{-1},
\end{split}
\end{equation*}
where the Cauchy--Schwarz inequality and the self-adjointness of $F' ( \phi )$ and $\Lambda^{-\frac12}$ were used.
\end{proof}

\section{Strong convergence rates of the spatial discretization}
\label{sec:spatial-analysis}

In this section, we analyze the strong approximation of the spatial discretization of the SPDE~\eqref{eq:abstract-SPDE} by a spectral Galerkin method.

It should be noted that essential difficulties exist when analyzing a finite element method in space for the considered SPDE. Indeed,
the dissipative property of the nonlinear mapping $P_h F$, where $P_h$ is the orthogonal projection, does not hold in the $\mathbb{H}^2$-norm.
This problem does not arise with the spectral Galerkin approximation that we now analyze.

To begin with, we consider a positive integer $N$ and define the finite dimensional subspace $U_N$
of $U=L^2 ( \mathcal{D} ; \R )$ by
$U_N := \text{span} \{ e_1, e_2, \cdots, e_N \}$, where we recall that $(e_j)_{j=1}^N$ are
the eigenfunctions of $\Lambda$. We next define the projection operator
$P_N: \dot{H}^{\alpha} \rightarrow U_N$ by
\begin{equation*}
P_N  \psi = \sum_{i=1}^N \langle \psi, e_i \rangle e_i,
\quad \forall ~ \psi \in \dot{H}^{\alpha}, \alpha \geq -1.
\end{equation*}
Moreover, the projection operator on $\H^{\beta}$, still denoted by $P_N$ for convenience, is given by \begin{equation*}
P_N X = [P_N x_1, P_N x_2]^T \,\,\,\, \mathrm{for} \,\,\,
X=[x_1,x_2]^T \in \H^{\beta},\beta \geq 0.
\end{equation*}
Then, one can immediately verify that
\begin{equation}\label{eq:P_N-estimate}
 \| P_N  \psi \| \leq \| \psi \|,
 \quad
 \| (I - P_N)  \psi \| \leq
    \lambda_N ^{-\frac \kappa 2} \| \psi \|_{\kappa},
    \quad \text{for} \quad \psi \in \dot{H}^{\kappa}, \kappa \geq 0.
\end{equation}
The discrete Laplacian $\Lambda_N: U_N \rightarrow U_N $ is defined by
\begin{equation*}
\Lambda_N \psi =  \Lambda P_N \psi = P_N \Lambda \psi
   = \sum_{i=1}^N \lambda_i \langle \psi, e_i \rangle e_i,
   \quad \forall~ \psi \in U_N.
\end{equation*}
Therefore, applying the spectral Galerkin method to the SPDE~\eqref{eq:abstract-SPDE} yields
the finite dimensional problem
\begin{equation}\label{eq:spectral-galerkin}
    \dd X^N(t) = A_N X^N(t) \, \dd t
       + {\bf F}_N (X^N(t)) \, \dd t   + B_N \dd W^Q(t)
\end{equation}
with the initial value $X^N(0)=\left[
\begin{array}{ccc}
     P_N u_0 \\
     P_N v_0
 \end{array}
\right]$. Here, we denote $F_N := P_N F$ and
\begin{equation*}
X^N =X^N(t) = \left[
 \begin{array}{ccc}
     u^N(t) \\ v^N(t)
 \end{array}
 \right]
,
\:
 A_N = \left[
 \begin{array}{ccc}
     0 & I \\
     -\Lambda_N & 0
 \end{array}
 \right]
,
\:
{\bf F}_N(X^N) = \left[
 \begin{array}{ccc}
     0 \\
     F_N (v^N)
 \end{array}
 \right]
,
\:
B_N = \left[
 \begin{array}{ccc}
     0 \\ P_N
 \end{array}
 \right].
\end{equation*}
Analogously to the continuous setting, the operator $A_N$ generates a strongly continuous semigroup $E_N(t)$ for $t \geq 0$ on $U_N \times U_N$, given by
\begin{equation*}
E_N(t)= e^{t A_N} =
\left[
 \begin{array}{ccc}
     C_N(t) &
     \Lambda_N^{-\frac12} S_N(t) \\
     -\Lambda_N^{\frac12} S_N(t) &
     C_N(t)
 \end{array}
  \right].
\end{equation*}
Obviously, the discrete cosine and sine operators
$C_N(t):= \cos( t \Lambda_N^{\frac12})$  and
$S_N(t):= \sin( t \Lambda_N^{\frac12})$ satisfy for $\psi \in \dot{H}^{\alpha},\alpha \geq -1$,
\begin{equation*}
C_N(t) P_N \psi = C(t) P_N \psi = P_N C(t) \psi, \,
S_N(t) P_N \psi = S(t) P_N \psi = P_N S(t) \psi.
\end{equation*}
Similarly, the mild solution of the semi-discrete problem~\eqref{eq:spectral-galerkin} reads
\begin{equation}\label{eq:mild-space}
X^N(t) = E_N(t) X^N(0) + \int_0^t E_N(t-s) {\bf F}_N(X^N(s)) \, \dd s
                 + \int_0^t E_N(t-s) B_N \, \dd W^Q(s).
\end{equation}

The following lemma, concerning the spatio-temporal regularity of $X^N(t)$, is crucial in the presented strong convergence analysis.
\begin{lemma}\label{lem:XN-regularity}
Under Assumptions~\ref{assum:nonlinearity} and~\ref{assum:noise-term} and assuming that
$ \| X_0 \|_{L^p(\Omega; \H^2)} < \infty$ for some $p \geq 2$, it holds that
\begin{equation}\label{eq:space-regularity-Z^N}
\sup_{t \in [0,T]} \| X^N(t) \|_{L^p(\Omega;\H^2)} < \infty.
 \end{equation}
In addition, for $0 \leq s \leq t \leq T$ and $\eta \in [1,2)$, it holds that
\begin{equation}\label{eq:time-continuity-Z^N}
 \| X^N(t) - X^N(s) \|_{L^p(\Omega;\H^\eta)}
  \leq \widehat{C}\big( t - s \big)^{\min\{2-\eta,\frac12\}}.
\end{equation}
\end{lemma}
\begin{proof}
The proof for the spatial regularity of the process $X^N(t)$ can be shown similarly to the proof of Theorem~\ref{thm:wellposedness-regularity-Z}. This is thus omitted. To prove the temporal H\"{o}lder regularity, we use the mild form of the semi-discrete solution~\eqref{eq:mild-space} and obtain
\begin{equation*}
\begin{split}
X^N(t) - X^N(s) &= \big( E_N(t-s) - I \big) X^N(s)
     \\& \quad + \int_s^t E_N(t-r) {\bf F}_N (X^N(r)) \,\dd r
     + \int_s^t E_N(t-r) B_N \,\dd W^Q(r).
\end{split}
\end{equation*}
Thus, by the triangle inequality, we get
\begin{equation*}
\begin{split}
\| X^N(t) -  X^N(s) \|_{L^p(\Omega; \H^{\eta})}
  & \leq \| ( E_N(t-s) - I ) X^N(s) \|_{L^p(\Omega; \H^{\eta})}
  \\& \quad + \Big\| \int_s^t E_N(t-r) {\bf F}_N(X^N(r)) \,\dd r
                          \Big\|_{L^p(\Omega; \H^{\eta})}
  \\& \quad + \Big\| \int_s^t E_N(t-r) B_N \,\dd W^Q(r)
                          \Big\|_{L^p(\Omega; \H^{\eta})}.
\end{split}
\end{equation*}
We now treat each of the above three terms separately.
The first term can be directly estimated by \eqref{eq:semigroup}
and \eqref{eq:space-regularity-Z^N} in order to get the estimates
\begin{equation*}
\big\| ( E_N(t-s) - I ) X^N(s) \big\|_{L^p(\Omega; \H^{\eta})}
\leq \widehat{C} \big( t - s \big)^{ 2- \eta}
          \| X^N(s) \|_{L^p(\Omega; \H^2)}
\leq \widehat{C} \big( t - s \big)^{ 2- \eta}.
\end{equation*}
For the second term, it follows from the stability of the semigroup, the bound~\eqref{eq:F-theta} with $\theta = \eta -1$, and the relation~\eqref{eq:space-regularity-Z^N} that
\begin{equation*}
\begin{split}
\Big\| \int_s^t & E_N(t-r) {\bf F}_N (X^N(r))
\,\dd r   \Big\|_{L^p(\Omega; \H^{\eta})}
\\ & \leq  \int_s^t \big\| E_N(t-r) {\bf F}_N (X^N(r)) \big\|_{L^p(\Omega; \H^{\eta})} \,\dd r
\\ & \leq \widehat{C}  \int_s^t \big\| {\bf F}_N (X^N(r)) \big\|_{L^p(\Omega; \H^{\eta})} \,\dd r
\\ & \leq  \widehat{C}
     \int_s^t \|  F(v^N(r)) \|_{L^p(\Omega; \dot{H}^{\eta-1})} \,\dd r
\\ & \leq  \widehat{C}
     \int_s^t \big( 1 +  \|  v^N(r)
      \|_{L^{\gamma p}(\Omega; \dot{H}^{1})}^{\gamma} \big)  \,\dd r
\\ & \leq \widehat{C} \big( t - s \big).
\end{split}
\end{equation*}
Finally, using the Burkholder--Davis--Gundy inequality, the stability of the sine and cosine operators as well as of the projection operator,
and the assumption~\eqref{eq:ass-AQ-condition} on the noise, we obtain
\begin{equation*}
\begin{split}
\Big\| & \int_s^t E_N(t-r) B_N \dd W^Q(r)
           \Big\|_{L^p(\Omega; \H^{\eta})}
\\ & \leq \left(
  \int_s^t
  \big\| S(t-r) \Lambda^{\frac{\eta-1}2}  \big\|_{\mathcal{L}_2^0}^2 +
  \big\|
 C(t-r) \Lambda^{\frac{\eta-1}2} \big\|_{\mathcal{L}_2^0}^2 \,\dd r
  \right)^{\frac12}
\\& \leq
\widehat{C} \,
\left( \int_s^t \big\|
  \Lambda^{\frac{\eta-2}2} \Lambda^{\frac12} Q^{\frac12}
  \big\|_{\mathcal{L}_2}^2 \, \dd r
  \right)^{\frac12}
\\& \leq \widehat{C} \big( t - s \big)^{\frac12}.
\end{split}
\end{equation*}
It remains to collect all the above estimates to conclude the proof.
\end{proof}

The following theorem gives the strong error, in the $\H^1$-norm, of the spatial approximation
of the stochastic wave equation~\eqref{eq:abstract-SPDE} by the spectral Galerkin method.

\begin{theorem}[Strong convergence rates of the spatial semi-discretization]\label{thm:space}
Let $X(t)$, resp. $X^N(t)$, denote the mild solution~\eqref{eq:mild-abstract} of the considered stochastic wave
equation with nonlinear damping, resp. the mild solution~\eqref{eq:mild-space} of its spectral Galerkin approximation.
Under Assumptions~\ref{assum:nonlinearity} and~\ref{assum:noise-term} and assuming that
$\|X_0\|_{L^p(\Omega;\H^2)} < \infty$ for some $p \ge 2$, we have for dimension $d=1$, the error estimate
\begin{equation*}
\big\|  X(t) - X^N(t) \big\|_{L^2( \Omega;\H^1 )}
\leq\widehat C \lambda_N^{-\frac12}.
\end{equation*}
For dimension $d=2$, for any sufficiently small parameter $\epsilon > 0$, the error estimate reads
\begin{equation*}
\big\|  X(t) - X^N(t) \big\|_{L^2( \Omega;\H^1 )}
\leq\widehat C \lambda_N^{-\frac12+\epsilon}.
\end{equation*}
\end{theorem}

\begin{proof}
To carry out the convergence analysis of the spectral Galerkin method,
we first define an auxiliary process $\widetilde{X}^N(t)$ satisfying
\begin{equation}\label{eq:auxiliary-process-Galerkin}
\dd \widetilde{X}^N(t)
    = A_N \widetilde{X}^N(t) \, \dd t
       + {\bf F}_N (X(t))\, \dd t   +  B_N  \dd W^Q(t),
       \quad
       \widetilde{X}^N(0)={X}^N(0),
\end{equation}
which can be regarded as a linear SPDE and its solution is given by
\begin{equation*}
    \widetilde{X}^N(t)
    =
    E_N(t) X^N(0)
    +
    \int_0^t  E_N(t-s) {\bf F}(X(s))\, \dd s
    +
    \int_0^t   E_N(t-s) B_N\,  \dd W^Q(s).
\end{equation*}
We can then split the error of the spatial approximation into two terms:
\begin{equation*}
\| X(t) - X^N(t) \|_{L^2(\Omega;\H^1)}
\leq
    \| X(t) - \widetilde{X}^N(t) \|_{L^2(\Omega;\H^1)}
   +  \| \widetilde{X}^N(t) - X^N(t) \|_{L^2(\Omega;\H^1)}.
\end{equation*}
For the first error term, we further divide it into three terms using the triangle inequality and the fact that $F_N=P_N F$.
For any $p\geq2$, we then obtain
\begin{equation*}
\begin{split}
\|  X(t) -\widetilde{X}^N(t) \|_{L^p(\Omega;\H^1)}
& \leq
      \| E(t) X(0) - E_N(t) X^N(0)\|_{L^p(\Omega;\H^1)}
\\& \quad
 + \Big\|  \int_0^t \left( E(t-s) - E_N(t-s) \right) {\bf F}(X(s))\, \dd s \Big\|_{L^p(\Omega;\H^1)}
\\& \quad
 + \Big\| \int_0^t  \left( E(t-s) B  - E_N(t-s) B_N \right)\,  \dd W^Q(s)
  \Big\|_{L^p(\Omega;\H^1)}
\\& =: I_1 + I_2 + I_3.
\end{split}
\end{equation*}
We now estimate the above three terms. For the first one, owing to the properties
of the projection operator, see equation~\eqref{eq:P_N-estimate}, and
using that $\|X(0)\|_{L^p(\Omega;\H^2)} < \infty$, one gets the estimate
\begin{equation*}
I_1 \leq\widehat C \, \lambda_N^{-\frac12}  \|X(0)\|_{L^p(\Omega;\H^2)}
    \leq\widehat C \, \lambda_N^{-\frac12}.
\end{equation*}
For the term $I_2$, we use again \eqref{eq:P_N-estimate}, the moment bound for the mild solution given in
\eqref{eq:moment-bounds-Z(t)}, and the assumption
\eqref{eq:ass-AQ-condition} on the noise to show
\begin{equation*}
\begin{split}
I_2 & = \Big\| ( I - P_N )  \int_0^t  E(t-s) {\bf F} (X(s))\, \dd s
                  \Big\|_{L^p(\Omega;\H^1)}
   \\ & = \Big\| ( I - P_N )
   \Big( X(t) - E(t) X(0) - \int_0^t E(t-s) B\, \dd W^Q(s) \Big)
                  \Big\|_{L^p(\Omega;\H^1)}
   \\ & \leq\widehat C \, \lambda_N^{-\frac12}
         \Big( \| X(t) \|_{L^p(\Omega;\H^2)}
         + \| X(0) \|_{L^p(\Omega;\H^2)}
         + \| \Lambda^{\frac 12} Q^{\frac12} \|_{\mathcal{L}_2}\Big)
   \\ & \leq\widehat C \, \lambda_N^{-\frac12}.
\end{split}
\end{equation*}
For the term $I_3$, the It\^{o} isometry, stability properties of the operators $C(t), S(t)$,
\eqref{eq:P_N-estimate} and
\eqref{eq:ass-AQ-condition} provide us with the estimate
\begin{equation*}
\begin{split}
 I_3 & \leq  \left(
   \int_0^t \left(\big\| S(t-s) - S_N(t-s) \big\|_{\mathcal{L}_2^0}^2
      + \big\| C(t-s) - C_N(t-s) \big\|_{\mathcal{L}_2^0}^2\right) \,\dd s
  \right)^{\frac12}
  \\& \leq \left( \int_0^t \big\|
       (I - P_N) \Lambda^{-\frac12} \Lambda^{\frac12} Q^{\frac12}
                \big\|_{\mathcal{L}_2}^2 \, \dd s \right)^{\frac12}
\\& \leq \widehat C \, \lambda_N^{-\frac12}.
\end{split}
\end{equation*}
Thus, combining the above, we obtain the bound
\begin{equation}\label{eq:space-rate-1}
\|  X(t) -\widetilde{X}^N(t) \|_{L^p(\Omega;\H^1)}
                          \leq\widehat C \lambda_N^{-\frac12}.
\end{equation}
To deal with the second error term
$\| \widetilde{X}^N(t) - X^N(t) \|_{L^2(\Omega;\H^1)}$,
we need the regularity of $\widetilde{X}^N(t)$ that we now show. By definition of the auxiliary process
and the use of our assumptions as well as the regularity of the mild solution of the stochastic wave equation,
we obtain the bound
\begin{equation*}
\begin{split}
\big\| \widetilde{X}^N(t) \big\|_{L^p(\Omega;\H^2)}
& \leq   \big\| E_N (t) X^N(0) \big\|_{L^p(\Omega;\H^2)}
  + \Big\| \int_0^t E_N(t-s) {\bf F} ( X(s) ) \,\dd s \Big\|_{L^p(\Omega;\H^2)}
\\ & \quad
+ \Big\| \int_0^t E_N(t-s) B_N \,\dd W^Q(s)  \Big\|_{L^p(\Omega;\H^2)}
\\ & \leq \big\| X(0) \big\|_{L^p(\Omega;\H^2)} +
  \Big\| \int_0^t E(t-s) {\bf F}(X(s)) \,\dd s
           \Big\|_{L^p(\Omega;\H^2)}
     + \big\| \Lambda^{\frac12} Q^{\frac12} \big\|_{\mathcal{L}_2}
\\& \leq\widehat C < \infty.
\end{split}
\end{equation*}
Next, applying the dissipative condition of the nonlinearity $F$ gives
\begin{equation*}\label{eq:space-derivative}
\begin{split}
 & \frac{\dd \left( \| \widetilde{X}^N(t) - X^N(t) \|_{\H^1}^2 \right)} {\dd t}
\\ & \, = 2 \langle \widetilde{X}^N(t) - X^N(t),
\tfrac{\dd ( \widetilde{X}^N(t) - X^N(t) )}{\dd t}
\rangle_{\H^1}
\\ & \, =
2 \langle \widetilde{X}^N(t) - X^N(t),
A_N ( \widetilde{X}^N(t) - X^N(t) )
\rangle_{\H^1}
\\ \qquad \qquad & \qquad
+ 2 \langle \widetilde{X}^N(t) - X^N(t),
{\bf F}_N ( X(t) )  -  {\bf F}_N (X^N(t) )
\rangle_{\H^1}
\\ &  \, =
2 \langle \widetilde{X}^N(t) - X^N(t),
{\bf F}_N (\widetilde{X}^N(t) )  - {\bf F}_N (X^N(t) )
\rangle_{\H^1}
\\ & \qquad
+ 2 \langle \widetilde{X}^N(t) - X^N(t),
{\bf F}_N (X(t)) - {\bf F}_N (\widetilde{X}^N(t) )
\rangle_{\H^1}
\\ & \, \leq
2~\widehat C~\big\|  \widetilde{X}^N(t) - X^N(t) \big\|_{\H^1}^2
+
2 \langle \widetilde{X}^N(t) - X^N(t),
{\bf F}_N (X(t)) - {\bf F}_N ( \widetilde{X}^N(t) )
\rangle_{\H^1}.
\end{split}
\end{equation*}
Therefore, in dimension $d=1$, integrating both sides of the above relation over $[0,t]$ and using the Sobolev embedding inequality
$\dot{H}^1 \subset V:=C(\mathcal D;\mathbb R)$, the regularity of the auxiliary and mild solutions,
one gets the estimate
\begin{equation*}
\begin{split}
\| \widetilde{X}^N & (t)  - X^N(t) \|_{\H^1}^2
\\& \leq \widehat C \int_0^t \| \widetilde{X}^N(s) - X^N(s) \|_{\H^1}^2 \, \dd s
 +\widehat C \int_0^t
   \| F_N ( v(s) ) - F_N (\widetilde{v}^N(s)) \|^2 \, \dd s
\\ & \leq
\widehat C \int_0^t \| \widetilde{X}^N(s) - X^N(s) \|_{\H^1}^2 \, \dd s
\\ &  \quad +
 \widehat C \int_0^t
 \left(  1 + \|\widetilde{v}^N(s)\|_V^{2(\gamma-1)} +
                \|v(s)\|_V^{2(\gamma-1)} \right)
 \| v(s) - \widetilde{v}^N(s) \|^2 \, \dd s
 \\ & \leq
\widehat C \int_0^t \| \widetilde{X}^N(s) - X^N(s) \|_{\H^1}^2 \, \dd s
\\ &  \quad +
 \widehat C \int_0^t
 \left(  1 + \|\widetilde{v}^N(s)\|_1^{2(\gamma-1)} +
                \|v(s)\|_1^{2(\gamma-1)} \right)
 \| X(s) - \widetilde{X}^N(s) \|_{\H^1}^2 \, \dd s.
\end{split}
\end{equation*}
Unfortunately, the embedding $\dot{H}^1 \subset V$ is not valid in dimension $d=2$.
As a result, the error analysis can only be done at the expense of an order reduction due to an infinitesimal factor.
Indeed, in this case and for $\epsilon>0$, we obtain the estimate
\begin{equation*}
\begin{split}
\| \widetilde{X}^N & (t)  - X^N(t) \|_{\H^1}^2
\\& \leq \widehat C \int_0^t \| \widetilde{X}^N(s) - X^N(s) \|_{\H^1}^2 \, \dd s
 +\widehat C \int_0^t
   \| \Lambda^{\epsilon} P_N  \Lambda^{-\epsilon}
         ( F ( v(s) ) - F (\widetilde{v}^N(s)) ) \|^2 \, \dd s
\\& \leq\widehat  C \int_0^t \| \widetilde{X}^N(s) - X^N(s) \|_{\H^1}^2 \, \dd s
 +\widehat C \, \lambda_N^{2\epsilon}   \, \int_0^t
    \big\|  ( F ( v(s) ) - F (\widetilde{v}^N(s)) )
       \big\|_{L^{\frac 2{1+2\epsilon}}}^2 \, \dd s
\\ & \leq
\widehat C \int_0^t \| \widetilde{X}^N(s) - X^N(s) \|_{\H^1}^2 \, \dd s
\\ &  \quad +
 \widehat C \, \lambda_N^{2\epsilon}   \, \int_0^t
 \left(  1
   + \|\widetilde{v}^N(s)\|_{L^{\frac {(\gamma-1)} {\epsilon}}}^{2(\gamma-1)}
    +  \|v(s)\|_{L^{\frac {(\gamma-1)} {\epsilon}}}^{2(\gamma-1)} \right)
 \| v(s) - \widetilde{v}^N(s) \|^2 \, \dd s
 \\ & \leq
\widehat C \int_0^t \| \widetilde{X}^N(s) - X^N(s) \|_{\H^1}^2 \, \dd s
\\ &  \quad +
 \widehat C \, \lambda_N^{2\epsilon}   \,  \int_0^t
 \left(  1 + \|\widetilde{v}^N(s)\|_1^{2(\gamma-1)} +
                \|v(s)\|_1^{2(\gamma-1)} \right)
 \| X(s) - \widetilde{X}^N(s) \|_{\H^1}^2 \, \dd s,
\end{split}
\end{equation*}
where the  inequality \eqref{eq:Lambda-minus-epsilon}, H\"{o}lder's inequality
$\|ab\|_{L^{\frac 2{1+2\epsilon}}}
    \leq \|a\|_{L^{\frac 1 \epsilon}} \|b\|_{L^2}$  and $\dot{H}^1 \subset L^{\frac {\gamma-1}{\epsilon}}$ for $d=2$  were used.

Finally, taking expectation in the above error estimates and using Gronwall's lemma, H\"{o}lder's inequality, the regularity of $X(s), \widetilde{X}^N(s)$ and equation~\eqref{eq:space-rate-1}, we obtain for dimension $d=1$, the error bound for the second error term
\begin{equation*}
\| \widetilde{X}^N(t) - X^N(t) \|_{L^2(\Omega;\H^1)}
\leq\widehat  C \, \lambda_N^{-\frac12}.
\end{equation*}
In dimension $d=2$, for a sufficiently small $\epsilon > 0$, we obtain the error bound for the second error term
\begin{equation*}
\| \widetilde{X}^N(t) - X^N(t) \|_{L^2(\Omega;\H^1)}
\leq\widehat  C \, \lambda_N^{-\frac12+\epsilon}.
\end{equation*}
This, together with the bound \eqref{eq:space-rate-1} for the first error term, finishes the proof.
\end{proof}

\section{Strong convergence rates of the full discretization}
\label{sec:full-analysis}
In this section, we analyze the strong convergence rate of the temporal discretization of the semi-discrete
problem~\eqref{eq:spectral-galerkin} coming from a spectral Galerkin approximation of the stochastic wave equation~\eqref{eq:abstract-SPDE}.

Let $ M \in \N $ and consider a uniform mesh
$ \{ t_0, t_1, \cdots, t_M \} $
of the interval $[0, T ]$
satisfying $t_m = m \tau$ with $\tau = \tfrac TM$ being the time stepsize.
To motivate the proposed time integrator, we observe that the mild solution of the
semi-discrete problem can be approximated as follows
\begin{equation*}
\begin{split}
X^N(t_1) & = E_N( t_1-t_0 ) X^N(t_0)
  + \int_{t_0}^{t_1} E_N( t_1-s ) {\bf F}_N
       ( X^N (s) )\, \dd s
\\ & \quad + \int_{t_0}^{t_1} E_N( t_1-s )
       B_N\, \dd W^Q(s)
\\ & \approx  E_N( \tau ) X^N(t_0)
      + \tau {\bf F}_N  ( X^N (t_1) )
      + E_N(\tau) B_N ( W^Q(t_1)-W^Q(t_0) ).
\end{split}
\end{equation*}
Hence, we define the fully discrete numerical solution by
\begin{equation}\label{eq:full-discretization}
X_{N,m+1} = E_N(\tau) X_{N,m}
   + \tau  {\bf F}_N(X_{N,m+1})
   + E_N(\tau) B_N \Delta W_m,
\end{equation}
where $\Delta W_m := W^Q(t_{m+1}) - W^Q(t_m)$.

\begin{proposition}
Let Assumption~\ref{assum:nonlinearity} hold and take a time stepsize verifying $C_1 \tau <1$.
Then, the implicit scheme~\eqref{eq:full-discretization} admits a unique solution in $U_N \times U_N$.
\end{proposition}
\begin{proof}
First, note that the nonlinear mapping ${\bf F}_N$ satisfies the monotonicity condition in $\H^1$:
\begin{equation*}
\langle  {\bf F}_N (X) - {\bf F}_N (Y), X - Y \rangle_{\H^1 }
  \leq C_1  \| X - Y \|_{\H^1 },
\quad
X, Y \in U_N \times U_N.
\end{equation*}
Second, introduce the continuous mapping $G_{N,\tau}(z) := z - \tau \mathbf{F}_N(z)$.
We then obtain, for any $z_1,z_2 \in U_N \times U_N$, that
\begin{equation*}
\langle  G_{N,\tau}(z_1) - G_{N,\tau}(z_2), z_1 - z_2 \rangle_{\mathbb{H}^1}
  \ge ( 1 - C_1 \tau ) \| z_1 - z_2 \|_{\mathbb{H}^1}.
\end{equation*}
Since $C_1 \tau < 1$, it follows from the uniform monotonicity theorem in \cite[Theorem C.2]{stuart1996dynamical} that the implicit scheme~\eqref{eq:full-discretization} is well-defined and a.s. uniquely solvable.
\end{proof}

To analyze the error of the above numerical scheme, we begin by rewriting the mild solution of the spectral Galerkin method~\eqref{eq:mild-space} as follows
\begin{equation}\label{eq:mild-space2}
\begin{split}
X^N(t_{m+1}) &= E_N(\tau) X^N(t_{m})
  + \int_{t_m}^{t_{m+1}} E_N(t_{m+1} - s) {\bf F}_N(X^N(s)) \,\dd s
  + \mathcal{O}_{N,t_m,t_{m+1}}
\\&=
E_N(\tau) X^N(t_{m})
  + \tau   {\bf F}_N(X^N(t_{m+1}))
               + E_N(\tau) B_N \Delta W_m + R_{m+1},
\end{split}
\end{equation}
where we define
\begin{equation*}
 \mathcal{O}_{N,t_m,t_{m+1}} :=
  \int_{t_m}^{t_{m+1}} E_N (t_{m+1}-s) B_N \, \dd W^Q(s)
\end{equation*}
and
\begin{equation*}
\begin{split}
R_{m+1}& := \int_{t_m}^{t_{m+1}}
\left( E_N(t_{m+1} - s) {\bf F}_N(X^N(s))
              - {\bf F}_N(X^N(t_{m+1})) \right) \,\dd s
\\ & \quad
   + \mathcal{O}_{N,t_m,t_{m+1}} - E_N(\tau) B_N \Delta W_m.
\end{split}
\end{equation*}

Subtracting the fully discrete solution~\eqref{eq:full-discretization} from the semi-discrete solution~\eqref{eq:mild-space2} yields
the following recursion for the temporal error
\begin{equation}\label{eq:err(N,M+1)}
e_{N,m+1} = E_N(\tau) e_{N,m} + \tau \big(
{\bf F}_N ( X^N(t_{m+1}) ) - {\bf F}_N ( X_{N,m+1} )  \big)
+ R_{m+1},
\end{equation}
where we define
\[
e_{N,m+1} := X^N(t_{m+1}) - X_{N,m+1}
.
\]

We are now fully prepared to prove an upper mean-square error bound for the temporal discretization of the semi-discrete
problem~\eqref{eq:spectral-galerkin}. This result plays a fundamental role in proving the strong convergence rate of the fully discrete scheme.

\begin{proposition}[Upper error bound]
\label{prop:upper-bounds}
Suppose that Assumptions \ref{assum:nonlinearity} and \ref{assum:noise-term} are fulfilled and let the time stepsize $\tau \leq \tfrac 1{6 ~ \rm{max}\{0,C_1\}}$.
Then, for any $m \in \{0,1,\cdots,M-1\}, M \in \N$, there exists a constant $\widehat C > 0$,
independent of $m$ and $N$, such that the error in the time discretization satisfies
\begin{equation*}
\E \Big[ \big\| e_{N,m+1} \big\|_{\H^1}^2 \Big]
\leq\widehat C \left (
\sum_{i=1}^{m+1}
\E \Big[ \big\| R_i \big\|_{\H^1}^2 \Big]
+ \tau^{-1} \sum_{i=0}^{m}
\E \Big[ \big\| \E[ R_{i+1}|\mathcal{F}_{t_{i}} ] \big\|_{\H^1}^2 \Big]  \right ).
\end{equation*}
\end{proposition}

\begin{proof}
By the definition of the error term $e_{N,m+1}$ in \eqref{eq:err(N,M+1)}, we know that,
for $m \in \{ 1, 2, \cdots, M-1 \}$,
\begin{equation*}
\begin{split}
\big\| e_{N,m+1} - &
\tau \big( {\bf F}_N(X^N(t_{m+1})) -
     {\bf F}_N(X_{N,m+1}) \big) \big\|_{\H^1}^2
= \big\| E_N(\tau) e_{N,m} + R_{m+1} \big\|_{\H^1}^2
\\&=
\big\| E_N(\tau) e_{N,m} \big\|_{\H^1}^2
+ \big\|  R_{m+1} \big\|_{\H^1}^2
+ 2 \left \langle E_N(\tau) e_{N,m} , R_{m+1} \right \rangle_{\H^1}.
\end{split}
\end{equation*}
Taking expectation on both sides of the above relation and using Young's inequality, one deduces
\begin{equation*}
\begin{split}
\E \Big[ \big\| & e_{N,m+1} -
\tau \big( {\bf F}_N ( X^N(t_{m+1}) ) -
       {\bf F}_N ( X_{N,m+1} ) \big) \big\|_{\H^1}^2 \Big]
\\&
= \E \Big[ \big\| E_N(\tau) e_{N,m} \big\|_{\H^1}^2 \Big]
+ \E \Big[ \big\|  R_{m+1} \big\|_{\H^1}^2 \Big]
+ 2 \E \Big[ \left \langle E_N(\tau) e_{N,m} , R_{m+1}
\right \rangle_{\H^1} \Big]
\\& \leq
\E \Big[ \big\| e_{N,m} \big\|_{\H^1}^2 \Big]
+ \E \Big[ \big\|  R_{m+1} \big\|_{\H^1}^2 \Big]
+ 2 \E \Big[ \left \langle E_N(\tau) e_{N,m} ,
\E \left[ R_{m+1}|\mathcal{F}_{t_m} \right] \right \rangle_{\H^1} \Big]
\\& \leq
\E \Big[ \big\| e_{N,m} \big\|_{\H^1}^2 \Big]
+ \E \Big[ \big\|  R_{m+1} \big\|_{\H^1}^2 \Big]
+ \tau \E \Big[  \big\| e_{N,m} \big\|_{\H^1}^2 \Big] +
\tau^{-1} \E \Big[ \big\|
\E \left[ R_{m+1}|\mathcal{F}_{t_m} \right] \big\|_{\H^1}^2 \Big],
\end{split}
\end{equation*}
where the stability of $E_N(\tau)$ and properties of conditional expectations were used.
Owing to the monotonicity of the nonlinearity $F$, one can then infer that
\begin{equation*}
\begin{split}
\E \Big[ \big\| & e_{N,m+1} -
\tau \big( {\bf F}_N ( X^N(t_{m+1}) )
          - {\bf F}_N ( X_{N,m+1} ) \big) \big\|_{\H^1}^2 \Big]
   \\&  =  \E \Big[ \big\| e_{N,m+1} \big\|_{\H^1}^2 \Big]
     + \tau^2  \E   \Big[ \big\|  {\bf F}_N ( X^N(t_{m+1}) )
           - {\bf F}_N ( X_{N,m+1} ) \big\|_{\H^1}^2 \Big]
   \\& \quad - 2 \tau \E \Big[ \left \langle
    e_{N,m+1}, {\bf F}_N ( X^N(t_{m+1}) )
       - {\bf F}_N ( X_{N,m+1} ) \right \rangle_{\H^1} \Big]
   \\& \geq ( { 1 - 2 C_1 \tau} )
      \E \Big[ \big\| e_{N,m+1} \big\|_{\H^1}^2 \Big].
\end{split}
\end{equation*}
Therefore, by iterating and noting that $e_{N,0}=0$, we get
\begin{equation*}
\begin{split}
\E \Big[ & \big\| e_{N,m+1} \big\|_{\H^1}^2 \Big]
\\& \leq ( { 1 - 2 C_1 \tau} )^{ - ( m + 1 )}
\E \Big[ \big\|  e_{N,0} \big\|_{\H^1}^2 \Big]
+ \sum_{i=1}^{m+1} ( { 1 - 2 C_1 \tau} )^{ - ( m + 2 -i )}
\E \Big[ \big\|  R_i \big\|_{\H^1}^2 \Big]
\\& \quad
+ \tau \sum_{i=0}^{m} ( { 1 - 2 C_1 \tau} )^{ - ( m + 1  -i )}
\E \Big[ \big\| e_{N,i} \big\|_{\H^1}^2 \Big]
\\& \quad + \tau^{-1} \sum_{i=0}^{m}
  ( { 1 - 2 C_1 \tau} )^{ - ( m + 1 -i )}
\E \Big[ \big\| \E \left[ R_{i+1}|\mathcal{F}_{t_i}\right] \big\|_{\H^1}^2 \Big]
\\& \leq\widehat
C \, \tau \sum_{i=0}^{m}
\E \Big[ \big\| e_{N,i} \big\|_{\H^1}^2 \Big]
+\widehat C \sum_{i=1}^{m+1}
\E \Big[ \big\| R_i \big\|_{\H^1}^2 \Big]
+\widehat C \,\tau^{-1} \sum_{i=0}^{m}
\E \Big[ \big\| \E \left[ R_{i+1}|\mathcal{F}_{t_i} \right] \big\|_{\H^1}^2 \Big],
\end{split}
\end{equation*}
where {we used $(1 - 2 C_1 \tau)^{-m} \leq 1$ for $C_1 < 0$
and $(1 - 2 C_1 \tau)^{-m} \leq (1+3C_1\tau)^m \leq e^{3C_1 T}$ for $C_1 > 0, \, \tau \leq \tfrac 1{6C_1}$.}
Finally, the use of Gronwall's inequality results in the desired assertion.
\end{proof}

Equipped with the above preparations, we now prove the mean-square convergence rates
of the temporal discretization.

\begin{theorem}[Strong convergence rates of the temporal semi-discretization]
\label{thm-main2}
Let Assumptions~\ref{assum:nonlinearity} and~\ref{assum:noise-term} be fulfilled
and assume that $\|X(0)\|_{L^p(\Omega;\H^2)} < \infty$ for some $p\ge2$.
Let $X^N(t)$ denote the mild solution of the semi-discrete problem~\eqref{eq:spectral-galerkin}
resulting from a spectral Galerkin discretization. Let $X_{N,m}$ denote the numerical approximation
in time by the modified implicit exponential Euler scheme~\eqref{eq:full-discretization}
with time stepsize $\tau = \tfrac TM$
{satisfying $\tau \leq \tfrac 1{6 ~ \rm{max}\{0,C_1\}}$.}
Then, there exists a positive constant $\widehat C$, independent of $N,M \in \N$, such that
\begin{equation*}
\sup_{0 \leq m \leq M}
\big\| X^N(t_m) -  X_{N,m}  \big\|_{L^2(\Omega;\H^1)}
\leq\widehat C \, \tau, \quad \text{for dimension} \quad d=1,
\end{equation*}
and for sufficiently small $\epsilon > 0$,
\begin{equation*}
\sup_{0 \leq m \leq M}
\big\| X^N(t_m) -  X_{N,m}  \big\|_{L^2(\Omega;\H^1)}
\leq\widehat C \, \tau^{1-\epsilon}, \quad \text{for dimension} \quad d=2.
\end{equation*}
\end{theorem}

\begin{proof}
We start by estimating the two terms
$\E \Big[ \big\| R_{i+1} \big\|_{\H^1}^2 \Big]$
and
$\E \Big[
\big\| \E \left[ R_{i+1}|\mathcal{F}_{t_i} \right] \big\|_{\H^1}^2
  \Big] $ on the right-hand side of the upper error bound given in Proposition~\ref{prop:upper-bounds}.
We first divide the term $\| R_{i+1} \|_{ L^2 ( \Omega; \H^1 ) }$
into three parts as follows,
\begin{equation}\label{eq:R-sepatation}
\begin{split}
\| R_{i+1} \|_{ L^2 ( \Omega; \H^1 ) }  & \leq
\Big\| \int_{t_i}^{t_{i+1}} \big( E_N(t_{i+1} - s) {\bf F}_N(X^N(s))
 - {\bf F}_N ( X^N(t_{i+1}) ) \big) \,\dd s \Big\|_{L^2(\Omega; \H^1)}
\\& \quad +
\Big\| \int_{t_i}^{t_{i+1}}  E_N( t_{i+1} - s)
   B_N \,\dd W^Q(s) -  E_N (\tau) B_N \Delta W_i
      \Big\|_{ L^2 ( \Omega; \H^1 ) }
\\ & \leq
 \Big\| \int_{t_i}^{t_{i+1}}
 \big( E_N (t_{i+1} - s) - I \big)  {\bf F}_N ( X^N (s) ) \,\dd s
\Big\|_{ L^2 ( \Omega; \H^1 ) }
\\& \quad +
\Big\| \int_{t_i}^{t_{i+1}}
\big( {\bf F}_N ( X^N(s) ) - {\bf F}_N ( X^N(t_{i+1}) ) \big) \,\dd s
\Big\|_{ L^2 ( \Omega; \H^1 ) }
\\& \quad +
\Big\| \int_{t_i}^{t_{i+1}}
\big[ E_N( t_{i+1} - s) - E_N( t_{i+1} - t_i) \big] B_N \,\dd W^Q(s)
\Big\|_{ L^2 ( \Omega; \H^1 ) }
\\& =: J_1 + J_2 + J_3.
\end{split}
\end{equation}
In dimension $d=1$.
Together with \eqref{eq:cosine-sine} in Lemma~\ref{lem:operator-estimate}, the assumptions on $F$, Lemma~\ref{lem:F-theta} and \eqref{eq:space-regularity-Z^N},
we arrive at
\begin{equation}\label{eq:J1}
\begin{split}
J_1  &=
 \Big\| \int_{t_i}^{t_{i+1}}
 \big( E_N(t_{i+1} - s) -I \big)  {\bf F}_N ( X^N (s) ) \,\dd s
 \Big\|_{ L^2 ( \Omega; \H^1 ) }
\\ & \leq
\int_{t_i}^{t_{i+1}}
 ( t_{i+1} - s)
 \big\|   F ( v^N (s) ) \big\|_{ L^2 ( \Omega; \dot{H}^1 ) } \,\dd s
\\& \leq\widehat C \tau^2.
\end{split}
\end{equation}
In dimension $d=2$, using \eqref{eq:F-theta} with $\theta = 1-\epsilon$ and Lemma \ref{lem:XN-regularity}, we arrive at the estimate
\begin{equation}\label{eq:J1'}
\begin{split}
J_1  &=
 \Big\| \int_{t_i}^{t_{i+1}}
 \big( E_N(t_{i+1} - s) -I \big)  {\bf F}_N ( X^N (s) ) \,\dd s
 \Big\|_{ L^2 ( \Omega; \H^1 ) }
\\ & \leq
\int_{t_i}^{t_{i+1}}
 ( t_{i+1} - s)^{ 1 - \epsilon }
 \big\|   F ( v^N (s) )
  \big\|_{ L^2 ( \Omega; \dot{H}^{1-\epsilon} ) } \,\dd s
\\& \leq\widehat C \,  \tau^{ 1 - \epsilon }   \int_{t_i}^{t_{i+1}}
  \big( 1 +   \| v^N (s)
      \|_{L^{2\gamma}( \Omega; \dot{H}^{1} )}^{\gamma}\big)\,\dd s
\\& \leq\widehat C \,  \tau^{ 2 - \epsilon }.
\end{split}
\end{equation}
For the second term $J_2$, in dimension $d=1$, we obtain
\begin{equation*}
\begin{split}
J_2  &=
 \Big\| \int_{t_i}^{t_{i+1}}
\big( {\bf F}_N (X^N(t_{i+1})) - {\bf F}_N (X^N(s)) \big)\, \dd s
\Big\|_{ L^2 ( \Omega; \H^1 ) }
\\& \leq
\int_{t_i}^{t_{i+1}} \Big\|F_N ( v^N(t_{i+1})) - F_N (v^N(s))
    \Big\|_{ L^2 ( \Omega;U)}\,   \dd s
\\& \leq\widehat C
\int_{t_i}^{t_{i+1}}
\left( 1+ \| v^N(s) \|_{L^{4(\gamma-1)}(\Omega;V)}^{\gamma-1}
   + \| v^N(t_{i+1}) \|_{L^{4(\gamma-1)}(\Omega;V)}^{\gamma-1} \right)
      \\& \qquad \times
   \| v^N (t_{i+1}) - v^N(s) \|_{ L^4 ( \Omega; U )}\, \dd s
\\& \leq\widehat C \tau^{\frac32}.
\end{split}
\end{equation*}
For the term $J_2$ in dimension $d=2$, we do a Taylor expansion of the nonlinearity, denote
$\xi(\lambda) := v^N(s) + \lambda (v^N(t_{i+1}) - v^N(s))$, and apply H\"{o}lder's inequality,
the Sobolev embedding inequality
$\dot{H}^{2\epsilon} \subset L^{\frac 2 {1-2\epsilon}}$, for a sufficiently  small $\epsilon$, Assumption~\ref{assum:nonlinearity} and
\eqref{eq:time-continuity-Z^N} to acquire the bound
\begin{equation}\label{eq:J2'}
\begin{split}
J_2  & \leq \int_{t_i}^{t_{i+1}}
\Big\| F ( v^N(t_{i+1})) - F (v^N(s)) \Big\|_{ L^2 ( \Omega; U )} \,\dd s
\\& = \int_{t_i}^{t_{i+1}} \Big\|
\int_0^1 F'(\xi(\lambda)) ( v^N(t_{i+1}) -  v^N(s) ) \,\dd \lambda
    \Big\|_{ L^2 ( \Omega; U )}   \,\dd s
\\& \leq\widehat C
\int_{t_i}^{t_{i+1}}  \int_0^1
   \Big\| \| F'(\xi(\lambda)) \|_{L^{\frac 1\epsilon}}  \cdot
          \| v^N(t_{i+1}) -  v^N(s) \|_{L^{\frac 2{1-2\epsilon}}}
     \Big\|_{ L^2 ( \Omega; \R )}
     \dd \lambda  \,\dd s
\\ & \leq\widehat C\int_{t_i}^{t_{i+1}}
\big( 1 + \sup_{s\in[0,T]}
\| v^N(s)\|_{L^{4(\gamma-1)}(\Omega;\dot{H}^1)}^{\gamma-1} \big) \cdot \| v^N (t_{i+1}) - v^N(s)
                 \|_{ L^4 ( \Omega; \dot{H}^{2\epsilon} )} \,\dd s
\\& \leq\widehat C \tau^{\frac32}.
\end{split}
\end{equation}
For the last term $J_3$, we apply It\^{o}'s isometry, \eqref{eq:cosine-sine} in
Lemma~\ref{lem:operator-estimate} (using the stability of the projection) and
the assumption on the noise \eqref{eq:ass-AQ-condition} to obtain the bound
\begin{equation}\label{eq:J3}
\begin{split}
J_3  &=
\Big\| \int_{t_i}^{t_{i+1}}
\big[ E_N( t_{i+1} - t_i) - E_N( t_{i+1} - s) \big] B_N \,\dd W^Q(s)
\Big\|_{ L^2 ( \Omega; \H^1 ) }
\\ &  \leq
\widehat C \Big( \int_{t_i}^{t_{i+1}}
\left(\big\| ( S_N ( t_{i+1} - t_i) - S_N ( t_{i+1} - s) )
\Lambda^{-\frac12} \Lambda^{\frac12} Q^{\frac12} \big\|_{\mathcal{L}_2}^2\right.
\\& \quad +\left.
\big\| ( C_N ( t_{i+1} - t_i) - C_N ( t_{i+1} - s) )
\Lambda^{-\frac12} \Lambda^{\frac12} Q^{\frac12} \big\|_{\mathcal{L}_2}^2\right)
\,\dd s
\Big)^{\frac12}
\\&  \leq\widehat C \bigg(
\int_{t_i}^{t_{i+1}}   (  s - t_i  )^2  \,\dd s
\bigg)^{\frac12}
\\&  \leq\widehat C \tau^{\frac32}.
\end{split}
\end{equation}
Collecting all the above estimates, we get the bound
\begin{equation}\label{eq:R-estimate}
\E \Big[ \big\| R_{i+1} \big\|_{\H^1}^2 \Big] \leq\widehat C \, \tau^3.
\end{equation}

It remains to estimate the term
$\E \Big[
    \big\| \E \left[ R_{i+1}|\mathcal{F}_{t_i} \right] \big\|_{\H^1}^2
    \Big]$.
First, observe that the stochastic integral vanishes under the conditional expectation.
Next, proceeding as in the proof of the estimates~\eqref{eq:R-sepatation}-\eqref{eq:J2'}, for dimension $d=1$,
we obtain the bound
\begin{equation*}
\begin{split}
\E \Big [ \big\| & \E \left[ R_{i+1} | \mathcal{F}_{t_i}\right]
               \big\|_{ \H^1 }^2 \Big ]
\\ & = \E \Big [ \Big\| \E \Big[ \int_{t_i}^{t_{i+1}}
\big[ E_N(t_{i+1} - s) {\bf F}_N(X^N(s))
- {\bf F}_N(X^N(t_{i+1})) \big] \,\dd s \big| \mathcal{F}_{t_i} \Big]
\Big\|_{  \H^1  }^2 \Big ]
\\& \leq
\E \Big [ \Big\|  \int_{t_i}^{t_{i+1}}
 \big( E_N(t_{i+1} - s) - I \big) {\bf F}_N( X^N (s) )  \,\dd s
 \Big\|_{  \H^1  }^2 \Big ]
\\& \quad
+ \E \Big [ \Big\| \E \Big[ \int_{t_i}^{t_{i+1}}
 \big( {\bf F}_N ( X^N(t_{i+1}) ) - {\bf F}_N ( X^N(s) ) \big) \,\dd s
 \big| \mathcal{F}_{t_i}  \Big]
\Big\|_{  \H^1  }^2 \Big ]
\\& \leq\widehat C \, \tau^4+
\E \Big [ \Big\| \E \Big[ \int_{t_i}^{t_{i+1}}
 \big( F_N ( v^N(t_{i+1}) ) - F_N ( v^N(s) ) \big) \,\dd s
 \big| \mathcal{F}_{t_i}  \Big] \Big\|^2 \Big ],
\end{split}
\end{equation*}
where \eqref{eq:J1} was used in the last inequality.
For dimension $d=2$, we obtain the bound
\begin{equation*}
\begin{split}
\E \Big [ \big\| & \E [ R_{i+1} | \mathcal{F}_{t_i}]
               \big\|_{ \H^1 }^2 \Big ]
\\& \leq\widehat C \, \tau^{4-2\epsilon}+
\E \Big [ \Big\| \E \Big[ \int_{t_i}^{t_{i+1}}
 \big( F_N ( v^N(t_{i+1}) ) - F_N ( v^N(s) ) \big) \,\dd s
 \big| \mathcal{F}_{t_i}  \Big] \Big\|^2 \Big ].
\end{split}
\end{equation*}
Applying a Taylor expansion to the nonlinearity $F_N$, we get,
for $s \in [ t_i , t_{i+1} ]$,
\begin{equation*}
\begin{split}
F_N ( v^N (t_{i+1}) ) & - F_N ( v^N(s) )
 = F_N'( v^N(s) ) (v^N(t_{i+1}) - v^N(s))
 \\& \quad + \int_0^1 F_N''(\chi(\lambda))
     \big ( v^N(t_{i+1}) - v^N(s),
      v^N(t_{i+1}) - v^N(s)  \big ) ( 1 - \lambda ) \,\dd \lambda,
\end{split}
\end{equation*}
where
$ \chi(\lambda) := v^N(s) + \lambda ( v^N(t_{i+1}) - v^N(s) ) $. The goal is now to estimate each term in this Taylor expansion.
First, one uses the definition of the mild solution~\eqref{eq:mild-space} of the semi-discrete problem and get the relation
\begin{equation*}
\begin{split}
v^N ( t_{i+1} ) - v^N ( s ) &=
  - \Lambda_N^{\frac12} S_N(t_{i+1} - s)  u^N ( s )
 + ( C_N(t_{i+1} - s) - I ) v^N ( s )
\\& \quad +
   \int_s^{t_{i+1}}  C_N(t_{i+1} - r)   F_N ( v^N(r) ) \,\dd r
   + \int_s^{t_{i+1}} C_N (t_{i+1} - r)  \,\dd W^Q(r)
\end{split}
\end{equation*}
which can then be inserted in the above Taylor expansion.
Next, owing to the property of stochastic integration and the conditional expectation, we get the relation
\begin{equation*}
\begin{split}
\E \Big[  & \int_{t_i}^{t_{i+1}}  F_N'( v^N (s) )
\int_s^{t_{i+1}} C_N( t_{i+1} - r )  \,\dd W^Q(r) \,\dd s
 \Big| \mathcal{F}_{t_i}  \Big]
\\&  = \int_{t_i}^{t_{i+1}} \E \Big[   \int_s^{t_{i+1}}
  F_N'( v^N (s) ) C_N( t_{i+1} - r )  \,\dd W^Q(r)
 \Big| \mathcal{F}_{t_i}  \Big]\, \dd s
\\&  = \int_{t_i}^{t_{i+1}} \E \Big[  \E \Big[  \int_s^{t_{i+1}}
  F_N'( v^N (s) ) C_N( t_{i+1} - r )  \,\dd W^Q(r)
  \Big| \mathcal{F}_s  \Big]
 \Big| \mathcal{F}_{t_i}  \Big]\, \dd s
\\&  = 0.
\end{split}
\end{equation*}
Next, by virtue of the triangle inequality and the property of the conditional expectation, we have
the decomposition
\begin{equation*}
\begin{split}
\E & \Big [ \Big\| \E  \Big[ \int_{t_i}^{t_{i+1}}
\big( F_N (v^N(t_{i+1})) - F_N (v^N(s)) \big) \,\dd s
 \big| \mathcal{F}_{t_i}  \Big] \Big\|^2 \Big ]
  \\ & \leq \Big\| \int_{t_i}^{t_{i+1}}
     F_N'(v^N(s))
    \big( - \Lambda_N^{\frac12} S_N(t_{i+1} - s) u^N ( s )
        \big)\,\dd s
     \Big\|_{L^2(\Omega;U)}^2
   \\ & \quad +  \Big\| \int_{t_i}^{t_{i+1}}
     F_N'(v^N(s)) ( C_N(t_{i+1} - s) - I ) v^N ( s ) \,\dd s
     \Big\|_{L^2(\Omega;U)}^2
   \\ & \quad + \Big\| \int_{t_i}^{t_{i+1}}
     F_N'(v^N(s)) \int_s^{t_{i+1}} C_N(t_{i+1}-r) F_N(v^N(r))
     \,\dd r \,\dd s   \Big\|_{L^2(\Omega;U)}^2
   \\ & \quad + \Big\| \int_{t_i}^{t_{i+1}} \!\!\!  \int_0^1 \!\!
     F_N''(\chi(\lambda)) \big( v^N(t_{i+1}) - v^N(s),
         v^N(t_{i+1}) - v^N(s) \big) ( 1 - \lambda )
     \,\dd \lambda \,\dd s   \Big\|_{L^2(\Omega;U)}^2
   \\& =: K_1 + K_2 + K_3 + K_4.
\end{split}
\end{equation*}
Our final task is to bound these four terms.

For the first term, $K_1$, one uses H\"{o}lder's inequality, \eqref{eq:cosine-sine} in Lemma~\ref{lem:operator-estimate} and
the spatial regularity of $X^N(t)$ to show that, in dimension $d=1$,
\begin{align*}
K_1 & = \Big\| \int_{t_i}^{t_{i+1}}
     F_N'(v^N(s)) ( S_N(t_{i+1}-s) - S_N(0) )
         \Lambda_N^{\frac12} u^N(s) \,\dd s
     \Big\|_{L^2(\Omega;U)}^2
  \\& \leq \left( \int_{t_i}^{t_{i+1}}  \big\|
  F_N'(v^N(s)) ( S_N(t_{i+1}-s) - S_N(0) )
         \Lambda_N^{\frac12} u^N(s) \big\|_{L^2(\Omega;U)}
         \,\dd s \right)^2
  \\&  \leq\widehat C \, \Big( \int_{t_i}^{t_{i+1}}
   \big( 1+\|v^N(s)\|_{L^{4(\gamma-1)}(\Omega;V)}^{\gamma-1} \big)
  \\& \qquad  \times
    \big\| ( S_N(t_{i+1}-s) - S_N(0) )
       \Lambda_N^{-\frac12} \Lambda_N u^N(s)
      \big\|_{L^4(\Omega;U)}   \,\dd s \Big)^2
    \\& \leq\widehat C \, \Big( \int_{t_i}^{t_{i+1}}
   \big( 1+\|v^N(s)\|_{L^{4(\gamma-1)}(\Omega;V)}^{\gamma-1} \big)
  \\& \qquad  \times
   \big\| ( S(t_{i+1}-s) - S(0) )
       \Lambda^{-\frac12} \Lambda u^N(s)
      \big\|_{L^4(\Omega;U)} \,\dd s \Big)^2
  \\& \leq\widehat C \, \tau^2 \left( \int_{t_i}^{t_{i+1}}
   \big( 1+\|v^N(s)\|_{L^{4(\gamma-1)}(\Omega;V)}^{\gamma-1} \big)
   \big\| u^N(s) \big\|_{L^4(\Omega;\dot{H}^2)}  \,\dd s \right)^2
  \\& \leq\widehat C \, \tau^4,
\end{align*}
where we used Assumption \ref{assum:nonlinearity} and H\"{o}lder's inequality in the second inequality.
For dimension $d=2$, using H\"older's inequality and \eqref{eq:sobolev}, it follows that,
for sufficiently small $\epsilon>0$,
\begin{equation*}
\begin{split}
K_1 &  \leq \left( \int_{t_i}^{t_{i+1}}  \big\|
  F_N'(v^N(s)) ( S_N(t_{i+1}-s) - S_N(0) )
         \Lambda_N^{\frac12} u^N(s) \big\|_{L^2(\Omega;U)}
         \,\dd s \right)^2
  \\& \leq\widehat C \, \Big( \int_{t_i}^{t_{i+1}}
   \Big\| \| F'(v^N(s)) \|_{L^{\frac 1\epsilon}}  \cdot
          \| ( S(t_{i+1}-s) - S(0) )
       \Lambda^{-\frac12}\Lambda u^N(s)
       \|_{L^{\frac 2{1-2\epsilon}}} \Big\|_{ L^2 ( \Omega; \R )} \,\dd s \Big)^2
  \\& \leq\widehat C \, \tau^{2-4\epsilon} \left( \int_{t_i}^{t_{i+1}}
   \big( 1+\|v^N(s)\|_{L^{4(\gamma-1)}(\Omega;\dot{H}^1)}^{\gamma-1} \big)
   \big\| u^N(s) \big\|_{L^4(\Omega;\dot{H}^2)}  \,\dd s \right)^2
  \\& \leq\widehat C \, \tau^{4-4\epsilon}.
\end{split}
\end{equation*}
Similarly to the above, for the second term in dimension $d=1$, one gets the bound
\begin{equation*}
K_2 \leq\widehat  C \, \tau^4.
\end{equation*}
And for dimension $d=2$, one obtains the bound
\begin{equation*}
K_2 \leq\widehat  C \, \tau^{4-4\epsilon}.
\end{equation*}
The estimate for the term $K_3$ follows similarly to the one for the term $K_1$,
but using the boundedness of the operator $C_N(t)$ and Lemma~\ref{lem:F-theta} instead.
We obtain the bound
\begin{equation*}
\begin{split}
K_3 & = \Big\| \int_{t_i}^{t_{i+1}}
     F_N'(v^N(s)) \int_s^{t_{i+1}} C_N(t_{i+1}-r) F_N(v^N(r))
     \,\dd r \, \dd s   \Big\|_{L^2(\Omega;U)}^2
  \\& \leq \left( \int_{t_i}^{t_{i+1}}  \Big\|
  F_N'(v^N(s)) \int_s^{t_{i+1}} C_N(t_{i+1}-r) F_N(v^N(r)) \dd r \Big\|_{L^2(\Omega;U)} \,\dd s \right)^2
  \\& \leq\widehat C \, \Big( \int_{t_i}^{t_{i+1}} \!\! \int_s^{t_{i+1}}\!\!
   \Big\| \| F'(v^N(s)) \|_{L^{\frac 1\epsilon}}  \cdot
          \| C_N(t_{i+1}-r) F_N(v^N(r))
       \|_{L^{\frac 2{1-2\epsilon}}} \Big\|_{ L^2 ( \Omega; \R )} \,\dd r \,  \,\dd s \Big)^2
  \\& \leq\widehat C \left( \int_{t_i}^{t_{i+1}} \int_s^{t_{i+1}}
   \big( 1+\|v^N(s)\|_{L^{4(\gamma-1)}(\Omega;\dot{H}^1)}^{\gamma-1} \big)
   \big\| F(v^N(r)) \big\|_{L^4(\Omega;\dot{H}^{2\epsilon})}
       \,\dd r \,  \dd s \right)^2
  \\& \leq\widehat C \, \tau^4.
\end{split}
\end{equation*}
For the last term $K_4$, first in dimension $d=1$, it follows from the Sobolev embedding inequality
$\dot{H}^{\frac12} \subset L^4$, the fact that $\dot{H}^{1} \subset V$ in dimension $d=1$,
and the spatio-temporal regularity of the numerical solution $v^N(t)$, that
\begin{equation*}
\begin{split}
K_4 &= \Big\| \int_{t_i}^{t_{i+1}}  \int_0^1
     F_N''(\chi(\lambda)) \big( v^N(t_{i+1}) - v^N(s),
         v^N(t_{i+1}) - v^N(s) \big)
         ( 1 - \lambda )
     \,\dd \lambda \, \dd s   \Big\|_{L^2(\Omega;U)}^2
  \\& \leq \left( \int_{t_i}^{t_{i+1}}  \int_0^1
    \Big\| F_N''(\chi(\lambda)) \big( v^N(t_{i+1}) - v^N(s),
         v^N(t_{i+1}) - v^N(s) \big) \Big\|_{L^2(\Omega;U)}
         \,\dd \lambda \, \dd s  \right)^2
  \\& \leq\widehat C \Big( \int_{t_i}^{t_{i+1}}  \int_0^1
  \big( 1+ \big\|v^N(s)+\lambda(v^N(t_{i+1}) - v^N(s)) \big\|_{L^{4(\gamma-2)}(\Omega;V)}^{\gamma-2} \big)
  \\& \qquad \times
 \big\| v^N(t_{i+1}) - v^N(s) \big\|_{L^8(\Omega;\dot{H}^{\frac12})}^2
   \,\dd \lambda \, \dd s \Big)^2
  \\& \leq\widehat C \Big( 1 + \sup_{s\in[0,T]} \|v^N(s)\|_{L^{4(\gamma-2)}(\Omega;\dot{H}^1)}^{\gamma-2}\Big)^2
 \cdot \Big( \int_{t_i}^{t_{i+1}}  ( t_{i+1} - s ) \,\dd s \Big)^2
  \\& \leq\widehat C \, \tau^4,
\end{split}
\end{equation*}
where we used Assumption \ref{assum:nonlinearity} and H\"{o}lder's inequality in the second inequality.
For dimension $d=2$, by the Sobolev embedding inequality
$\dot{H}^{\frac{1+\epsilon}{2+\epsilon}} \subset L^{4+2\epsilon}$, one gets
\begin{equation*}
\begin{split}
K_4 & \leq \left( \int_{t_i}^{t_{i+1}}  \int_0^1
    \Big\| F_N''(\chi(\lambda)) \big( v^N(t_{i+1}) - v^N(s),
         v^N(t_{i+1}) - v^N(s) \big) \Big\|_{L^2(\Omega;U)}
         \,\dd \lambda \, \dd s  \right)^2
  \\& \leq\widehat C \Big( \int_{t_i}^{t_{i+1}}  \int_0^1
  \Big\| \|F_N''(\chi(\lambda))\|_{L^{\frac 4 \epsilon + 2}}
\cdot \| v^N(t_{i+1}) - v^N(s)\|_{L^{4+2\epsilon}}^2
 \Big\|_{L^2(\Omega;\R)}
   \,\dd \lambda \, \dd s \Big)^2
  \\& \leq\widehat C \Big( 1 + \sup_{s\in[0,T]} \|v^N(s)\|_{L^{4(\gamma-2)}(\Omega;\dot{H}^1)}^{\gamma-2}\Big)^2
 \cdot \Big( \int_{t_i}^{t_{i+1}} \!\!\!
 \| v^N(t_{i+1}) - v^N(s)
 \|_{L^8(\Omega;\dot{H}^{\frac{1+\epsilon}{2+\epsilon}})} ^2
 \,\dd s \Big)^2
  \\& \leq\widehat C \, \Big( \int_{t_i}^{t_{i+1}}
 ( t_{i+1} - s )^{\frac 2{2+\epsilon}} \,\dd s \Big)^2
 \\& \leq\widehat C \, \tau^{ 4 - \frac {2\epsilon}{2+\epsilon}}.
\end{split}
\end{equation*}
Collecting all the above estimates, for dimension $d=1$, one arrives at the bound
\begin{equation*}
\E \Big[
\big\| \E [ R_{i+1}|\mathcal{F}_{t_i} ] \big\|_{\H^1}^2
  \Big]
  \leq\widehat C \, \tau^4.
\end{equation*}
For dimension $d=2$, one gets the bound
\begin{equation*}
\E \Big[
\big\| \E [ R_{i+1}|\mathcal{F}_{t_i} ] \big\|_{\H^1}^2
  \Big]
  \leq\widehat C \, \tau^{4-4\epsilon}.
\end{equation*}
The above estimates, in conjunction with the bound~\eqref{eq:R-estimate} and Proposition~\ref{prop:upper-bounds},
finish the proof of the theorem.
\end{proof}

At last, gathering Theorem \ref{thm:space} with Theorem \ref{thm-main2},
we get the strong convergence rates of the fully discrete scheme \eqref{eq:full-discretization}.
\begin{corollary}[Strong convergence rates of the full discretization]
Under the setting of  Theorem \ref{thm:space} and Theorem \ref{thm-main2}, there exists a positive constant $\widehat C$, independent of $N,M \in \N$, such that
\begin{equation*}
\sup_{0 \leq m \leq M}
\big\| X(t_m) -  X_{N,m}  \big\|_{L^2(\Omega;\H^1)}
\leq \widehat C \lambda_N^{ -\frac12 } + \widehat C \, \tau \quad \text{for dimension} \quad d=1,
\end{equation*}
and for sufficiently small $\epsilon > 0$,
\begin{equation*}
\sup_{0 \leq m \leq M}
\big\| X(t_m) -  X_{N,m}  \big\|_{L^2(\Omega;\H^1)}
\leq \widehat C \lambda_N^{ -\frac12 + \epsilon} +
\widehat C \, \tau^{1-\epsilon}, \quad \text{for dimension} \quad d=2.
\end{equation*}
\end{corollary}

\section{Numerical experiments}\label{sec;Numerical experiments}

We conclude this paper by illustrating the above theoretical findings with numerical experiments.
Let us consider the stochastic wave equation with nonlinear damping
\begin{equation}\label{numerical-example}
\begin{split}
\left\{
    \begin{array}{lll}
    \dd u = v\, \dd t,
    \\
    \dd v = \Delta u\, \dd t + (v - v^3)\, \dd t + \dd W^Q(t),
    \quad \text{in} \,\, \mathcal{D} \times (0, 1]
    \end{array}\right.
\end{split}
\end{equation}
with $\mathcal{D} = (0,1)^d$, for $d=1,2$, equipped with homogeneous Dirichlet boundary condition. The Fourier coefficients of the initial positions are randomly set to $0$ or $1$ and the obtained vector is then divided by the eigenvalues of the Laplacian. The initial velocity is set to be $v(0)=0$.
The covariance operators of the infinite-dimensional Wiener process
${W^Q(t)}$, for $t \in [0,1]$, are chosen as $Q=\Lambda^{-1.005-d/2}$.
In what follows, we use the fully discrete scheme \eqref{eq:full-discretization} to approximate solution to the SPDE~\eqref{numerical-example}.
The strong error bounds are measured in the mean-square sense and
the expectations are approximated by computing averages over $1000$ samples.
We have checked empirically that the Monte--Carlo errors are negligable.

\begin{figure}[!htb]
  \centering
\begin{varwidth}[t]{\textwidth}
  \includegraphics[width=2.5in]{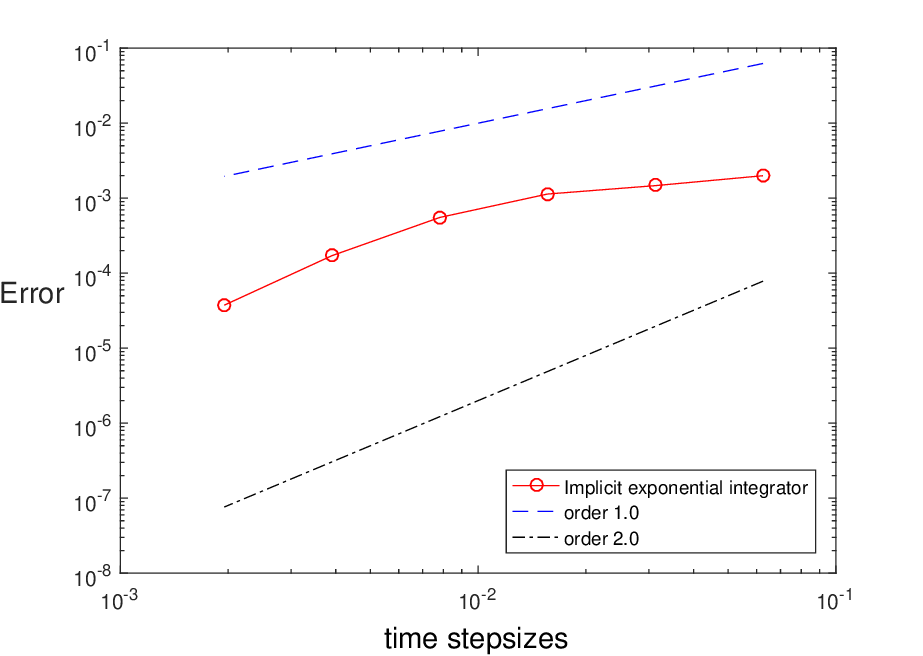}
 \end{varwidth}
 \quad
 \begin{varwidth}[t]{\textwidth}
  \includegraphics[width=2.5in]{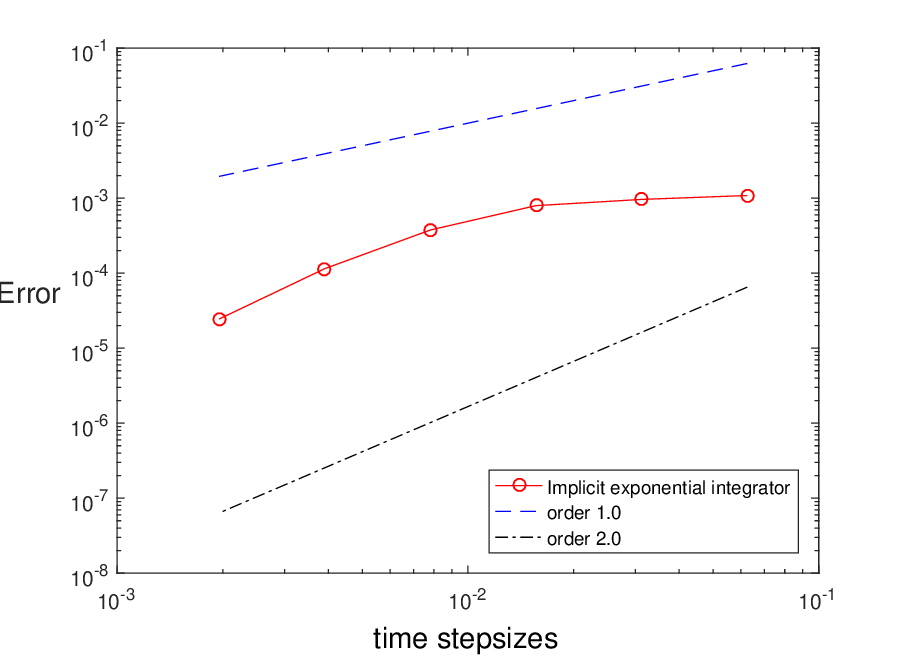}
 \end{varwidth}
  \caption{Strong convergence rates of the temporal discretization given by the implicit exponential Euler scheme (IEE) (left: $d=1$, right: $d=2$).}
 \label{F1}
\end{figure}

We now fix $N = 100$ for $d=1$ and $N = 30$ for $d=2$ and investigate the strong convergence rates in time by using the stepsizes $\tau = 2^{-j},j=4,5,\ldots,9$.
The reference solution is computed numerically by using the proposed time integrator with the reference stepsize $\tau= 2^{-10}$. In the loglog plots from Figure~\ref{F1}, one can observe that the approximation errors of the implicit exponential Euler scheme decrease with order $1$.
This is consistent with our theoretical findings.

\begin{figure}[!htb]
  \centering
\begin{varwidth}[t]{\textwidth}
  \includegraphics[width=2.5in]{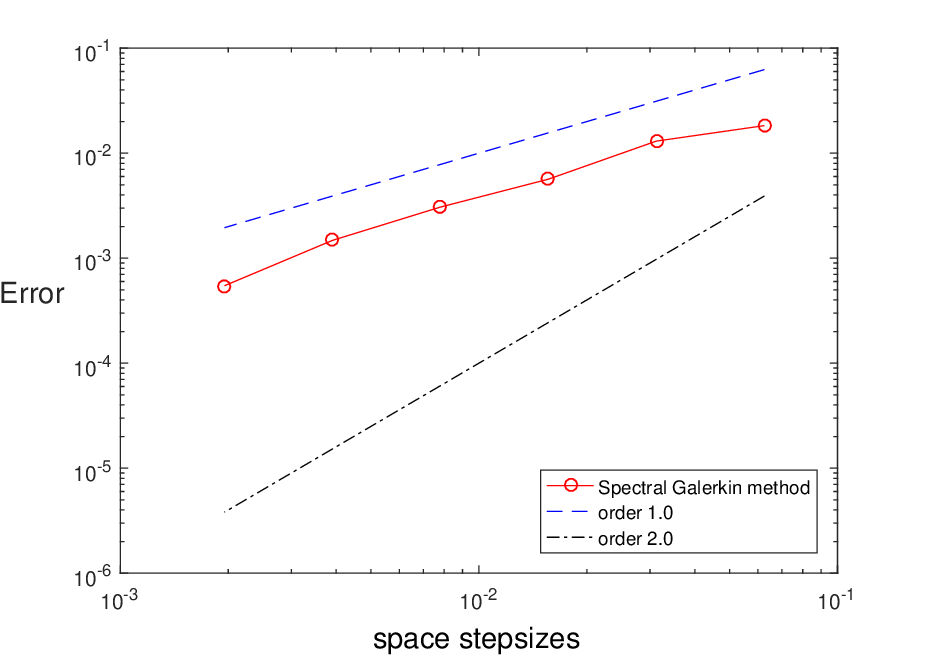}
 \end{varwidth}
  \caption{Strong convergence rates of the spatial discretization given by the spectral Galerkin method.}
 \label{F2}
\end{figure}

To visually illustrate the error in space for $d=1$, we compute a reference solution by using the proposed numerical scheme with $\tau=2^{-5}$ and $N=2^{10}$.
The resulting errors of six different mesh parameters
$N = 2^i, i = 4,5,\ldots,9$ are plotted in Figure~\ref{F2}
on a log--log scale.
One can observe that the expected convergence rates agree with those indicated in Theorem~\ref{thm:space}.

\section*{Conclusion and perspectives}

This work provides a comprehensive analysis of the strong convergence rates for the full discretization of the stochastic wave equation (SWE)
with nonlinear damping in dimension one and two.
The stochastic wave equation is discretized in space by a spectral Galerkin method and in time by a modified implicit Euler scheme.
The main idea is to achieve upper mean-square error bounds (cf. Proposition \ref{prop:upper-bounds}),
which enables us to prove mean-square convergence rates for the considered SWE without requiring an a priori high-order moment estimates of the fully discrete solution.
We prove that the strong rates of convergence are $1$ in time and $\frac12$ in space, in dimension one. In dimension two, the analysis is more involved and
the strong rates of convergence are shown to be $1-$ in time and $\frac12-$ in space. Beside these theoretical insights, the paper also offers a numerical validation of the proved convergence estimates.

Some fundamental questions are left open for possible future works. Let us mention a few of them:
\begin{enumerate}
\item We have only considered stochastic wave equation perturbed by additive noise in $d=1,2$. 
     The error analysis for the space dimension $d = 3$ or higher is non-trivial, more efforts should be paid to develop new techniques to solve it.
     It may also be interesting to study more complex problems, for instance problems with multiplicative noise.
    Models driven by multiplicative noise are more involved, in particular in the presence of a non-globally Lipschitz nonlinearity.
    We are not aware of well-posedness results for such problems or a convergence analysis of a numerical scheme in the existing literature.

\item In real applications, with complicated domains for instance, it may be difficult or impossible to find the eigenvalues and eigenfunctions of the Laplace operator $\Lambda$.
    In such cases, one is perhaps not able to use a spectral Galerkin approximation.
    The finite element method (FEM) could then be used to open a new path to numerical simulations.
    This numerical scheme enjoys several advantages as, for example, the possibility to consider SWEs with irregular domains.
    However, the error analysis of the FEM for SWEs is much more involved than for the spectral Galerkin method in this work.
    In particular, when our approach for deriving moment bounds of numerical solutions is applied to the finite element method,
    essential difficulties occur: the finite element projection $P_h$ destroys some of the dissipativity properties of the nonlinearity.
    A deeper numerical analysis of the FEM when applied to SWEs is thus needed.

\item 
The present scheme requires the computation of exponential matrix and some iteration which might be costly and hence further work on other scheme is important.
The literature offers many advanced numerical techniques for treating non-globally Lipschitz and superlinearly growing nonlinearities.
    One of them is the taming technique used for parabolic problems (see for example \cite{brehier2022weak,cai2021weak,cai2023strong,wang2024linearly,wang2020efficient} and references therein).
    To develop and analyze taming techniques in the context of time integrators for the stochastic wave equation with nonlinear damping
    could possibly also provide efficient numerical schemes for such SPDEs.
\end{enumerate}

\end{document}